\documentclass[11pt]{article}
\usepackage[margin=1in]{geometry}
\usepackage{amsmath, amssymb, amsthm, bm}
\usepackage{graphicx}
\usepackage{authblk}
\usepackage{soul, color}
\usepackage{hyperref}
\usepackage{natbib}
\usepackage{booktabs}
\usepackage{algorithm}
\usepackage{algpseudocode}
\usepackage{comment} 
\usepackage{amsfonts} 
\usepackage{upgreek}
\usepackage{setspace}

\newif\ifjasa
\jasafalse

\newtheorem{theorem}{Theorem}
\newtheorem{lemma}{Lemma}
\newtheorem{proposition}{Proposition}
\newtheorem{corollary}{Corollary}

\newtheorem{definition}{Definition}
\newtheorem{assumption}{Assumption}

\newcommand{\Input}{\item[\textbf{Input:}]}
\newcommand{\Output}{\item[\textbf{Output:}]}

\newcommand{\conv}{\operatorname{conv}}
\newcommand{\cone}{\operatorname{cone}}

\newcommand{\calH}{\mathcal{H}}

\newcommand{\calU}{\mathcal{U}}
\newcommand{\calV}{\mathcal{V}}

\newcommand{\R}{\mathbb{R}}

\newcommand{\papertitle}{Geometric Identification and Consistent Estimation of Source Apportionment with Application to Greenland Summit Aerosols
}
\newenvironment{keywords}{%
  \par\small
  \noindent\textbf{Keywords}: \ignorespaces
}{\par}
\graphicspath{ {./figure/} }

\title{\textbf{\papertitle}}
\author{Bora Jin$^{1,*}$, Nicholas J. Spada$^{2}$, Jessie Creamean$^{3,4}$, Matthew D. Shupe$^{5,6,7}$, Abhirup Datta$^{8}$\\
$^{1}$Department of Applied Mathematics and Statistics, Colorado School of Mines \\
$^{2}$Air Quality Research Center, University of California, Davis \\
$^{3}$Department of Atmospheric Science, Colorado State University \\
$^{4}$Department of Environmental Science, Stockholm University \\
$^{5}$Cooperative Institute for Research in Environmental Sciences/ National Snow and Ice Data Center, University of Colorado \\
$^{6}$NOAA Physical Sciences Laboratory \\ 
$^{7}$Department of Meteorology, Stockholm University \\
$^{8}$Department of Biostatistics, Johns Hopkins University \\
$^{*}$bora.jin@mines.edu }

\date{}

\begin{document}
\maketitle

\begin{abstract}
Over the last four decades, the Arctic has warmed nearly four times faster than the rest of the planet, with anthropogenic aerosols among the major forcings. Greenland Summit Station has collected size-resolved particulate matter year-round since 2003, yet these records remain underutilized. Attributing multipollutant measurements to their sources (source apportionment) can be cast as a non-negative matrix factorization (NMF) problem. Because NMF is non-unique, source apportionment imposes often unverifiable constraints such as structural sparsity or rotations tuned using prior scientific knowledge, which is scarce for remote regions like the Arctic. Geometric NMF approaches offer a more data-driven route to identification, but many suffer from scaling ambiguity, assume exact separability, and lack a statistical framework for consistent estimation. We address these limitations by formally establishing statistical identifiability of the source attribution matrix under a stochastic framework that replaces hard separability with soft probabilistic relaxations. We then present a scalable geometric algorithm to estimate this matrix and prove its consistency, to our knowledge the first such result requiring no exact sparsity, no parametric distributional assumptions, and accommodating spatio-temporal dependence. Analyzing the 2003--2016 Summit record with our framework and HYSPLIT trajectories, we characterize long-range transport of anthropogenic and natural species into the Arctic.

\end{abstract}


\section{Introduction} \label{sec:intro}


Over the last four decades, the Arctic has warmed four times faster than the rest of the planet \citep{rantanen_arctic_2022}. Anthropogenic aerosols are among the forcings shown to drive this phenomenon, called Arctic amplification \citep{wu_stronger_2024}. The warming Arctic alters atmospheric and oceanic dynamics, as changes in one aspect of the Earth system inevitably propagate to others. For example, feedbacks linking marine biology, aerosol, and cloud formation have long been hypothesized \citep{charlson_oceanic_1987}. 


Aerosols are the largest source of uncertainty in estimates of climate forcings, largely because of their morphological and chemical complexity \citep{forster_earths_2023}. Depending on its composition, size, and degree of aging, the same population of particles can exert either a warming or a cooling influence. Mineral dust, for example, scatters shortwave radiation \citep{brahney_dust_2024} and in the Arctic can originate from regional sources. During transport, mineral dust can acquire coatings of sulfate, nitrate, or hydrocarbons from anthropogenic pollution, which alter its scattering and its ability to activate and glaciate cloud droplets.

Identifying the varying sources of Arctic aerosol particle populations is therefore essential, since knowing the sources and their contributions ultimately connects to the direct and indirect radiative effects of those populations within the Arctic system \citep{schmale_aerosols_2021}. Source apportionment addresses this by quantifying which sources contribute to each measured species, attributing observed aerosol concentrations to underlying sources. Knowing the source location is also important to understand underlying composition and spatiotemporal development, which is why source apportionment is often assisted by wind direction or trajectory analyses to recover the geographic origin of the pollutants. Collectively, such insights from source apportionment can support efforts to better understand and protect the delicate Arctic environment. 


Multivariate receptor models are the standard tool for source apportionment. Let $Y_{ij}$ denote the concentration of the $j$th of $J$ pollutants in the $i$th of $n$ data records, where $i$ indexes a time point or a space-and-time combination. With $K$ sources, we model
\begin{equation}\label{eq:main}
    Y_{ij} = \sum_{k=1}^K W_{ik} H_{kj},
\end{equation}
where $H_{kj} \geq 0$ is the concentration of pollutant $j$ per unit of emissions from source $k$, and $W_{ik} \geq 0$ is the intensity of source $k$ in record $i$, which can be thought of as its emission volume. We leave `one unit' vague, as our estimands of interest are invariant to the unit choice. Writing $Y_i = (Y_{i1},\ldots,Y_{iJ})^\top \in \R_+^J$ and $W_i = (W_{i1},\ldots,W_{iK})^\top \in \R_+^K$, we have $Y_i^\top = W_i^\top H$ with $H = (H_{kj}) \in \R_+^{K\times J}$ the source profile matrix. Stacking the $n$ records gives $Y = WH$, with $Y \in \R_+^{n\times J}$ and $W \in \R_+^{n\times K}$. Using multipollutant measurements alone, source apportionment is therefore a non-negative matrix factorization (NMF) problem. See \cite{krall_statistical_2019} for a comprehensive review of source apportionment methods. 
We summarize the methods most relevant to our work below; an extended review is given in Supplementary Material \ref{subsec:rev}.

Many algebraic methods minimize ordinary or weighted least squares $\|Y - WH\|^2$. The Chemical Mass Balance model \citep{watson_chemical_1991} requires known $H$. When both $H$ and $W$ are unknown, they are estimated jointly by standard NMF implementations \citep{gaujoux_flexible_2010} and by positive matrix factorization \citep[PMF;][]{paatero_positive_1994}, one of the most widely used source apportionment algorithm. Identifiability of least-squares-based NMF (LS-NMF) typically impose exact sparsity on the profile matrix $H$, such as at least $K-1$ zeroes per row or a unique tracer pollutant for each source in data \citep{sug_park_bilinear_2002}. Sources rarely emit exactly zero of a pollutant, and unique tracers are rarely available without chemically speciated data. Existing asymptotic results require strong additional assumptions such as identically and independently distributed (iid) rows of $W$ \citep{yang_confirmatory_1994} and apply only to global optima \citep{park_bilinear_1999}, which LS-NMF algorithms are not guaranteed to find, as the loss function, with both $W$ and $H$ unknown, is non-convex.

Probabilistic approaches assume a stochastic data-generating model, from early work with iid source intensities and a single unknown source \citep{bandeen-roche_source_1991} to models of temporal \citep{heaton_incorporating_2010, rusanen_novel_2023}, spatial \citep{jun_multivariate_2013}, or spatio-temporal \citep{frigeri_bayesian_2025} dependence in environmental data. Estimation proceeds via Bayesian MCMC \citep{park_multivariate_2001, hackstadt_bayesian_2014, frigeri_bayesian_2025} or maximum likelihood with bootstrapping \citep{christensen_accounting_2002}. Identifiability is achieved through parametric forms, such as Gaussian or autoregressive processes, and through prior choices, and to our knowledge there is little asymptotic theory on the robustness of inference under misspecification of the parametric assumptions.

Geometric NMF approaches exploit the fact that each $Y_i = W_i^\top H$ can be written as $Y_i  = \sum_{k=1}^K W_{ik}h_k$, where $h^{\top}_k$ is the $k$th row of $H$. So, each $Y_i$ lies in the cone whose corner lines (`extremal rays') are the rows of $H$, because the $W_{ik}$'s are non-negative. Equivalently, the row-normalized data $Y^*_i$ lies in the polytope (simplex) with vertices at the row-normalized rows of $H$. The extremal rays of the cone (or equivalently, the vertices of the simplex), and hence $H$ up to scaling, can thus be identified from the shape of the point cloud, provided that data $\{Y_i\}_{i=1}^n$ offers sufficient geometric coverage of the cone. After estimating $H$, the weights $W$ can be obtained by non-negative least squares \citep{kumar_fast_2013}.

\cite{henry_history_1997} formalized this idea for source apportionment, leading to the Unmix multivariate receptor model \citep{henry_multivariate_2003}. \cite{donoho_when_2003} establishes identifiability conditions in terms of the conical geometry. Many geometric methods rely on separability: for each source there exist pure observations, in which only that source is active, so that some $Y_i$'s lie exactly on the extremal rays of the cone. End-member extraction algorithms from remote sensing, such as vertex component analysis \citep[VCA;][]{nascimento_vertex_2005}, XRAY \citep{kumar_fast_2013}, and N-FINDR \citep{winter_n-findr_1999}, exploit separability to find pure observations. Near-separability \citep{gillis_robust_2014, kumar_near-separable_2015} and the sufficiently scattered condition \citep{fu_identifiability_2018} relax separability, but they remain deterministic conditions on the observed sample. 
Other geometric identifiability results \citep{gillis_fast_2014, arora_computing_2016} address when a unique factorization can be recovered efficiently, not whether a sample-based estimator converges to the true estimand. They also ignore the spatio-temporal dependence that environmental data actually exhibit. 


These methods, whether algebraic, probabilistic, or geometric, leave three gaps. First, NMF is generally non-identifiable, and the sparsity patterns or deterministic geometric conditions such as separability imposed to restore identifiability are difficulty to verify, sensitive to misspecification, and unlikely to hold in finite environmental samples and lacking physical justification \citep{viana_source_2008, ma_signal_2014}. Second, asymptotic theory is largely absent. Geometric methods do not posit a stochastic data-generating model, and probabilistic models that do typically lack accompanying consistency results. Third, even identified factors are determined only up to scaling. Resolving the scale with domain knowledge is difficult in newer geographies like the Arctic where potential sources are poorly known. The common remedy, row-normalizing $H$, is problematic when pollutants are measured in incompatible units or a pollutant that dominates a source is unmeasured. 

In this manuscript we synthesize probabilistic and geometric approaches that have rarely overlapped to develop a theory of identifiability and consistent estimation for source apportionment. We introduce a stochastic data-generating framework that replaces deterministic geometric assumptions with probabilistic relaxations better suited to environmental data. We shift the inferential target from the non-unique factor matrices to the source attribution matrix $\Phi$, whose $(k,j)$th entry is the fraction of the expected concentration of pollutant $j$ attributable to source $k$. This estimand is invariant to NMF scale ambiguity and directly answers the policy-relevant question of which sources drive each pollutant. It also remains meaningful when pollutants are measured in incompatible units or when some pollutants are not measured, and is robust to sample-specific errors and pollutant-specific instrumental scalings. We establish identifiability and consistent estimation of $\Phi$ under stationarity and ergodicity of the source intensity processes, without parametric distributional assumptions and allowing for spatio-temporal dependence typical of environmental data. We give an explicit algorithm, GeomNMF, that consistently estimates $\Phi$ via a global search given sufficient computation, with scalable variants for large, high-dimensional data. We apply our framework to a multi-year record of size- and time-resolved aerosol at Greenland Summit Station, a remote, high-altitude site well suited to capturing long-range aerosol transport into the Arctic. Despite pervasive low-concentration observations that can defeat common practices in the widely used PMF, GeomNMF recovers physically interpretable, well-supported sources spanning industrial emissions, smelters, vessels, long-range mineral dust, and sea salt, and identifies a Titanium (Ti)-rich crustal source traceable to the Canadian Shield not previously reported for Greenland.


\section{Arctic aerosol at Greenland Summit} \label{sec:data}

Greenland Summit Station (72.6$^\circ$N, 38.5$^\circ$W; hereafter Summit) is the only high-altitude, inland, year-round observing station in the Arctic. At 3200 m above sea level and far from any local source, Summit integrates aerosol delivered by long-range transport from the free troposphere, so its composition reflects Northern Hemisphere emissions as they are transported over Greenland. Continuous sampling at Summit is difficult, since the site is remote and the weather is severe. Before a size-resolved aerosol sampling program operated at Summit from 2003 to 2020, aerosol monitoring in Greenland was limited to short-term, discontinuous campaigns \citep{vancuren_aerosols_2012}. The resulting nearly two-decade record, however, remains underutilized; consequently, which sources supply Summit aerosol, in what proportions, and from where has not been resolved from a long, continuous record.

Samples from the Summit program were collected with a Davis Rotating-drum Unit for Monitoring (DRUM) impactor, which segregates particles into eight size bins spanning 0.09 to 10 $\upmu\mathrm{m}$ diameter. Fine particulate mass from the DRUM has shown excellent agreement with collocated Federal Reference Method measurements \citep{cahill_seasonal_2014}. Elemental concentrations ($\mathrm{ng\,m^{-3}}$) were quantified by synchrotron-induced X-ray fluorescence for 29 species: Na, Mg, Al, Si, P, S, Cl, K, Ca, Ti, V, Cr, Mn, Fe, Co, Ni, Cu, Zn, Ga, As, Se, Br, Rb, Sr, Y, Zr, Mo, Pb, and non-volatile H. The daily-resolution dataset \citep{spada_size-resolved_2021} is publicly available through the U.S. National Science Foundation's \href{https://arcticdata.io/catalog/view/doi%3A10.18739%2FA28W3837S}{Arctic Data Center}.

Aerosol reaching Summit reflects natural emissions, such as mineral dust and sea spray, and anthropogenic emissions from combustion, smelting, and shipping. Distinguishing the two is central to Arctic aerosol science \citep{moschos_equal_2022}. Natural emissions are themselves climatically significant and must be characterized as a baseline for detecting anthropogenic contributions, which are the mitigable component. We therefore address two questions: (i) for each measured element, what fraction of its concentration at Summit is attributable to each source, and which of these sources are natural and which anthropogenic; and (ii) where these sources originate. Question (i) is a source apportionment problem and the focus of our methodology. Question (ii) is addressed by combining the estimated source intensities with back-trajectory analysis. The size segregation of the DRUM record provides a physical basis for separating natural from anthropogenic contributions in question (i): anthropogenic aerosols predominantly occupy the fine mode, whereas mineral dust and sea spray typically occupy the coarse mode \citep{seinfeld_atmospheric_2006, quinn_arctic_2007}. Even with this physical basis, two features of Summit make question (i) difficult for existing methods. First, because the sources reaching Summit are not well characterized \textit{a priori}, there are no reliable source profiles with which to fix the scale of the factors, and the answer must be expressed in a scale-free form. Second, without local, high-concentration signatures, the resolvable signals are weak, and below-detection observations are pervasive. Section \ref{sec:GL_prep} explains how our approach handles these challenges.

\section{Geometric identification in probabilistic source apportionment}\label{sec:identification}

\subsection{Source attribution matrix}

In source apportionment, we advocate focusing on quantities invariant under the scaling indeterminacy. An important quantity that satisfies this scale invariance is the proportion of each pollutant attributable to each source. We now define this quantity and subsequently show that it can be consistently estimated under mild assumptions using geometric methods. Since $\sum_{i=1}^nW_{ik}H_{kj}$ is the total concentration of pollutant $j$ from source $k$ in the data, the \textit{proportion} of total concentration of pollutant $j$ \textit{attributable to source $k$} 
is 
\begin{equation}\label{eq:sample_phi}
    \frac{\sum_{i=1}^n W_{ik}H_{kj}}{\sum_{\ell = 1}^K\sum_{i=1}^n W_{i\ell}H_{\ell j}},
\end{equation}
which corresponds to the `percent of species' output produced by the PMF software \citep{us_epa_epa_2015}, although we have not found this exact formula explicitly documented. We introduce a population-level analog of this proportion in \eqref{eq:sample_phi}. Under a stochastic framework for the intensities $W$ and mild conditions formally defined later, $\frac 1n \sum_{i=1}^n W_{ik} \to \mu_k$ where $\mu_k=\mathbb{E}(W_{ik})$ is the expected intensity from the $k$th source. So, the proportion in \eqref{eq:sample_phi} is the sample version of the population-level quantity 
\begin{equation}\label{eq:contrib}
\phi_{kj}=\frac {\mu_k H_{kj}}{\sum_{\ell=1}^K \mu_{\ell} H_{\ell j}}.
\end{equation}
The $K \times J$ column-stochastic matrix $\Phi=(\phi_{kj})$ is our estimand, representing the proportion of the total concentration of pollutant $j$ attributed to source $k$. We call $\Phi$ the source attribution matrix.
 
Unlike a row-normalized $H$, $\Phi$ is normalized within pollutants rather than within sources. It therefore does not purport to describe a source's full emissions profile; its entries retain the interpretation of source-specific contributions to each measured pollutant. Thus, the definition remains meaningful even when some of a source's dominant pollutants are unmeasured. Furthermore, $\Phi$ is invariant to pollutant-specific scale or unit changes. From \eqref{eq:contrib}, scaling the $j$th column of $Y$ by $\sigma_j>0$ gives $\phi_{kj}^{\text{(scaled)}} = \mu_k H_{kj} \sigma_j/(\sum_{\ell=1}^K \mu_{\ell} H_{\ell j} \sigma_j) = \phi_{kj}$. If $Y=WH$ is rewritten as $Y=\widetilde W \widetilde H$ where $\widetilde W = WD$ and $\widetilde H =D^{-1}H$ for some diagonal matrix $D$ with positive entries, then $\Phi$ computed from $(W,H)$ is the same as that from $(\widetilde W, \widetilde H)$. This scale invariance is expected and desirable: unit changes do not affect source attributions. 

\subsection{Statistical identifiability} 

Beyond scale invariance, $\Phi$ is statistically identifiable. Our argument uses the conical geometry of the data cloud. We begin by introducing the main geometric concepts.

\begin{definition}[Conical hull and extremal rays] Let $\calV=\{v_1,\ldots,v_M\} \subseteq \mathbb R^d_+$. Then $\cone(\calV)=\{\sum_{m=1}^M \alpha_m v_m : \alpha_m \geq 0, v_m \in \calV\}$ is its conical hull. A set $\calU=\{u_1,\ldots,u_R\} \subseteq \cone(\calV)$ is said to be a set of representatives of extremal rays of $\cone(\calV)$ if (i) $\cone(\calU)=\cone(\calV)$ and (ii) $u_i \notin \cone\{ u_j : j \neq i\}$ for all $i$.
\end{definition}

The extremal rays are the corner lines (the one-dimensional faces) of the polyhedral cone $\cone(\calV)$. None of these rays can be expressed as a non-negative linear combination of the others, and thus they form the minimal generating set whose conical hull is $\cone(\calV)$. 

Set $\calH=\{h_1,\dots,h_K\}$ where $h_k^\top$ is the $k$th row vector of $H$. Without loss of generality, we can assume that none of the $h_{k}$'s are non-negative linear combinations of the other $h_{k'}$ (equivalently, $H$ has full row rank); otherwise, we can remove that source. If $Y_{i}$ is generated according to \eqref{eq:main}, each vector $Y_i = \sum_{k=1}^K W_{ik}h_k$ lies in $\cone(\calH)$. Figure \ref{fig:cloud} (left) illustrates the point cloud (grey vectors) lying in the cone of three extremal rays (red vectors). 

\begin{figure}
    \centering
    \includegraphics[width=0.85\linewidth]{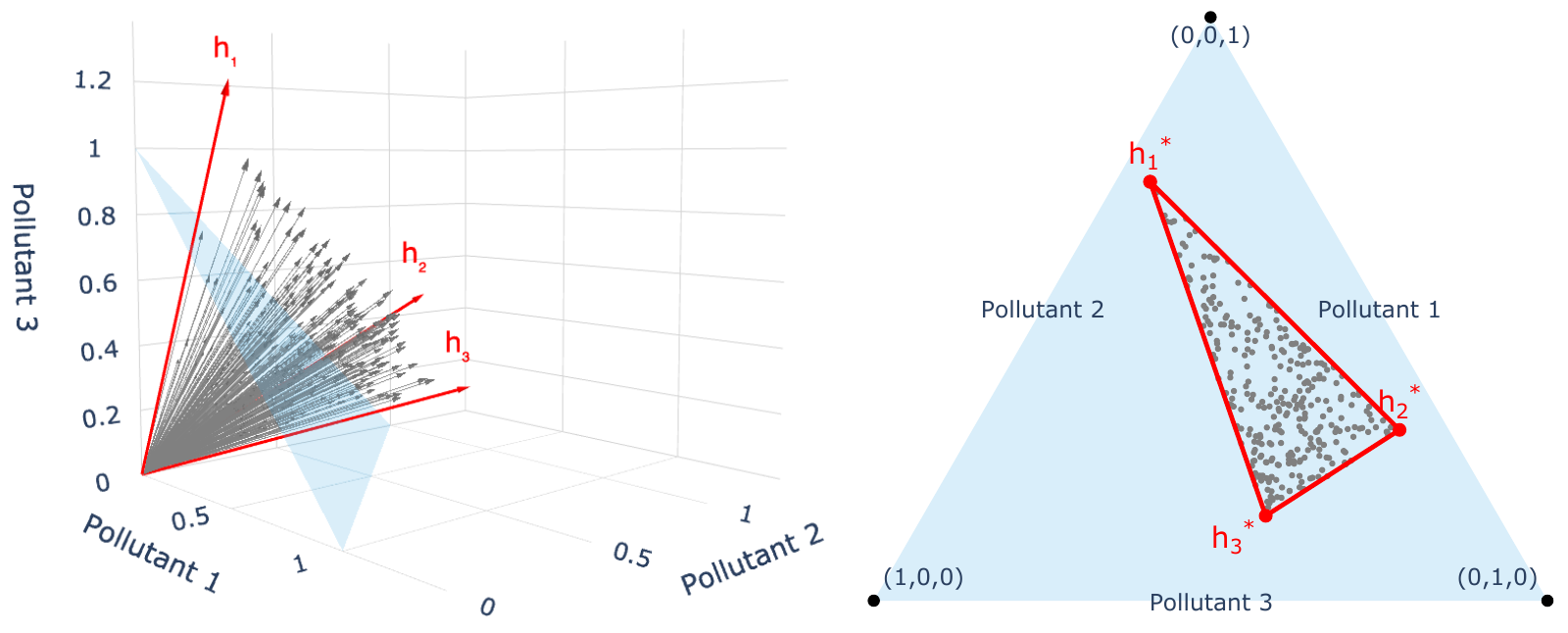}
    \caption{Geometric representation of multipollutant source apportionment (three pollutants, three sources). Left: gray arrows are rows of $Y$ (data points), and red arrows are rows of $H$. Right: the corresponding normalized vectors, shown as gray and red dots.}
    \label{fig:cloud}
\end{figure}

To establish identification and consistency results, we specify data-generating assumptions. Following prior probabilistic approaches \citep{bandeen-roche_source_1991, park_multivariate_2001, jun_multivariate_2013}, we treat $H$ as fixed and $\{W_i\}_{i=1}^n$ as random by specifying assumptions on their distributions. We show in Section \ref{sec:w} that in the noiseless case, recovery of $\{W_i\}_{i=1}^n$ up to column-wise scalings is possible, not just estimation of their distributions.

\begin{assumption}[Ergodicity and stationarity of source intensities]\label{asm:erg}
    $\{W_i\}_{i\geq 1}$ is an ergodic stationary process on $\mathbb R_+^K$ with marginal distribution $\mathbb P_W$ supported on a set $\mathbb S_W \subseteq \mathbb R_+^K$ and with finite mean $\mu=\mathbb E(W_i)$. 
\end{assumption} 

Requiring $\{W_i\}_{i \geq 1}$ to be only stationary and ergodic, rather than iid, allows for temporal or spatio-temporal dependence. We also impose no parametric assumptions on $\mathbb P_W$ or on the dependence structure of the processes. This assumption covers many standard time-series models, including stable autoregressive (AR), moving-average (MA), and ARMA processes with a stable AR component. 

We further impose the following condition on the support set $\mathbb S_W$ of the source intensities:
\begin{assumption}[Probabilistic separability]\label{asm:sep} 
There exist constants $c_k > 0$ such that $c_k e_k \in \mathbb S_W$ for every $k=1,\ldots,K$, where $e_k$ is the $k$th canonical basis vector in $\R^K$.
\end{assumption} 
This assumption ensures that, with positive probability, the process $\{W_i\}_{i\geq 1}$ produces realizations where one coordinate dominates. Unlike the traditional deterministic separability discussed in Section \ref{sec:intro}, which requires exact zeros in the $W$ matrix, this probabilistic version only requires the support of the distribution of $W$ to contain points arbitrarily close to each of the co-ordinate axes. This relaxation is well-suited for source apportionment, where it is plausible for one source to occasionally dominate but unrealistic for all others to be completely absent. Under these two assumptions, we state the identifiability result.  

\begin{theorem}[Statistical identifiability]\label{thm:welldefn}
Suppose each data point $Y_i$ admits two NMF representations $Y_i^\top=W_i^\top H=\widetilde W_i^\top \widetilde H$, where $H$ and $\widetilde H$ are $K\times J$ non-negative matrices, \(\{W_i\}_{i\ge1}\) and \(\{\widetilde W_i\}_{i\ge1}\) are $K$-dimensional processes satisfying Assumptions \ref{asm:erg} -- \ref{asm:sep}, and at least one of $H$ or $\widetilde H$ has full row rank. Let $\Phi$, defined in \eqref{eq:contrib}, denote the source attribution matrix derived from $\mathbb P_W$ and $H$, and let $\widetilde\Phi$ be the analogous quantity based on $\mathbb P_{\widetilde W}$ and $\widetilde H$. Then $\widetilde \Phi = \Phi$ up to a permutation of rows. 
\end{theorem}
Hence $\Phi$ is well defined and statistically identifiable up to reordering of the sources, even though the NMF representation of the data matrix $Y$ need not be unique.

\subsection{Measurement error} \label{sec:error}

The identifiability of $\Phi$ remains unchanged under separable multiplicative measurement error. Both additive and multiplicative error structures have been considered in source apportionment \citep{krall_statistical_2019}. Because non-negativity forces additive errors to depend on the signal, we adopt multiplicative error, $Y_{ij} = \epsilon_{ij} \gamma_{ij}$ with $\epsilon_{ij}>0$, where $\gamma_{ij}=\sum_{k=1}^KW_{ik}H_{kj}$ is the noise-free signal. This model includes proportional error, $\epsilon_{ij} = 1+\rho_{ij}$ with mean-zero relative error $\rho_{ij}$, as a special case. 

We impose a separable error structure, $\epsilon_{ij}=a_ib_j$, for $a_i > 0,\; b_j > 0,\;i\ge1,\; j=1,\ldots,J$. The separable error model, $Y_{ij}=a_i b_j\sum_{k=1}^K W_{ik}H_{kj}$, has an intuitive interpretation. The error component $a_i$ represents sample-specific random error shared across pollutants, such as meteorology-driven dilution or ventilation effects (e.g., precipitation, boundary layer height). The component $b_j$ represents pollutant-specific scaling, such as pollutant-specific sensor scaling noise. 

We treat the $b_j$'s as fixed, reflecting that such scaling effects are typically sensor-dependent and stable. However, the identifiability result below also extends to random $b_j$'s, provided they are independent of $\{a_i\}_{i\ge 1}$ and $\{W_i\}_{i\ge 1}$ jointly. Since the decomposition $a_ib_j = (a_i/c)(cb_j)$ is unchanged for any $c>0$, we adopt $\mathbb{E}(a_i)=1$ for identifiability. We further assume $\{a_i\}_{i \ge 1}$ is independent of $\{W_i\}_{i \ge 1}$. 
 
Let $D_b:=\text{diag}(b_1,\ldots,b_J)$ and define the reparameterization $W'_i := a_i W_i\in\mathbb R_+^K$ and $H{'} := H D_b\in\mathbb R_+^{K\times J}.$ Then for every $i$, $Y_i^\top = a_i W_i^\top H D_b = (W'_i)^\top H'$, so $\{Y_i\}_{i \ge 1}$ admits an NMF representation with $\{W'_i\}_{i \ge 1}$ and $H'$. Corollary \ref{cor:logerr} in Supplementary Material \ref{sec:pf} shows that when $\{a_i\}_{i \ge 1}$ is a stationary and weakly mixing positive process with $\mathbb E(a_i)=1$, independent of $\{W_i\}_{i\ge 1}$, the source attribution matrix from $\mathbb P_{W'}$ and $H'$ is identifiable up to row permutation and equals $\Phi$, the source attribution matrix from $\mathbb P_{W}$ and $H$. Hence, both the identifiability and the value of $\Phi$ are preserved under multiplicative separable measurement error, without assuming any parametric form for the error distribution.

\section{Consistent estimation}\label{sec:est}

\subsection{Convex hull and vertex estimation} 

We now provide an algorithm to consistently estimate $\Phi$. Figure \ref{fig:cloud} (left) illustrates that with sufficiently rich sampling, the cone would fill up, and one should be able to identify the corner vectors $h_k$ (up to scalar multiples) from the data. To visualize the identification of these extremal rays, we intersect the cone with the canonical simplex and show the barycentric projections in Figure \ref{fig:cloud} (right). The resulting scatter is contained within the triangle with vertices $h_1^*, h_2^*, h_3^*$, where each $h_k^*$ is the normalized version of $h_k$. Since $\Phi$ is invariant to row-scaling of $H$ by Theorem \ref{thm:welldefn}, we estimate the row-normalized $H$ and use it to consistently estimate $\Phi$. 

Let $d_k >0$ be the sum of the entries in $h_k$, and set $D_h=\text{diag}(d_1,\ldots,d_K)$. Define the row-stochastic matrix $H^*=D_h^{-1}H$. Then we can write $Y=\widetilde W H^*$ where $\widetilde W= WD_h$. Let $r=(r_1,\ldots,r_n)^\top=\widetilde W 1_K$ be the vector of row sums of $\widetilde W$ where $1_K$ is the $K$-vector of ones. Set $W^*_i=\widetilde W_i/r_i$ and form $W^*$ by stacking the row vectors $W_i^{*\top}$. Let $Y^* := W^*H^*$. Since both $W^*$ and $H^*$ are row-stochastic, $Y^*$ is also row-stochastic. Moreover, $Y=(\text{diag}(r))W^*H^* = (\text{diag}(r)) Y^*$, yielding $Y1_J = r$. Hence, $r$ can be computed directly from the observed data $Y$ as the row-sums. We therefore estimate $H^*$ based on the row-normalized data $Y^* = (\text{diag}(r))^{-1}Y$.

Let $\Delta^{J-1} = \{ v \in \R^J: v_j \geq 0, \sum_{j=1}^J v_j =1 \}$ denote the canonical simplex in $\R^J$. Each $Y_i^* = \sum_{k=1}^K W^*_{ik}h^*_{k}$ is a convex combination of the rows $h^*_{k}$ of $H^*$ with weights $W^*_{ik}$. Hence, every $Y_i^*$ belongs to the convex hull generated by the rows of $H^*$, $\conv(\calH^*) = \conv\{h^*_1,\dots,h^*_K\}$. Geometrically, $\conv(\calH^*)$ is a convex polytope and can be viewed as the intersection of $\cone(\calH)$ with $\Delta^{J-1}$, as illustrated in Figure \ref{fig:cloud}. When $W_i^*$ lies near a vertex of $\Delta^{K-1}$, the corresponding $Y^*_i$ is close to one of the vertices $h^*_k$. By Assumption \ref{asm:sep}, $\mathbb P_{W^*}$ has support near every corner of $\Delta^{K-1}$, implying that as $n$ increases the sample $\{Y_i^*\}_{i=1}^n$ will reveal all vertices of $\conv(\calH^*)$, which are exactly the rows of $H^*$. Thus, $H^*$ can be estimated by identifying the extreme points of the point cloud formed by the rows of $Y^*$. Formally, Theorem \ref{thm:hull} below states that the sample convex hull of $Y^*$ consistently estimates $\conv(\calH^*)$. Throughout, $\to_{a.s.}$ and $\to_{p}$ denote almost sure and in-probability convergence, respectively.

\begin{theorem}[Hausdorff consistency of the sample convex hull]\label{thm:hull}
Suppose $Y_i^\top = W_i^\top H$ where $H$ is a $K \times J$ non-negative matrix of full row rank, 
and $\{W_i\}_{i\ge 1}$ is a $K$-dimensional process satisfying Assumptions \ref{asm:erg} -- \ref{asm:sep}. Let $Y^*_i$ denote the normalized $Y_i$ and $H^*$ the row-normalized $H$. Then, 
\ifjasa
$d_{\mathrm{Haus}}\big(\conv\{Y^*_1,\ldots,Y^*_n\},\ \conv\{h_{1}^*,\ldots,h_{K}^*\}\big)\to_{a.s.} 0$
\else
$$
d_{\mathrm{Haus}}\big(\conv\{Y^*_1,\ldots,Y^*_n\},\ \conv\{h_{1}^*,\ldots,h_{K}^*\}\big)\to_{a.s.} 0
$$ 
\fi 
where $d_{\mathrm{Haus}}$ is the Hausdorff distance.
\end{theorem}

Previous consistency results for sample convex hulls assume iid observations \citep{brunel_uniform_2019}, or for dependent stationary sequences, additional regularity such as $\beta$-mixing \citep{kallel_topological_2024} or a positive minimal index \citep{louhichi_sharp_2025}. Theorem \ref{thm:hull} requires only stationarity and ergodicity, realistic assumptions for source apportionment. 

We now use this convex hull consistency result to recover $H^*$ and then $\Phi$. As the true $H^*$ has $K$ rows, the set conv$(\calH^*)$ is a $K$-vertex polytope. Since the sample convex hull converges to this polytope, a well-chosen set of $K$ points from the sample hull will converge to the rows of $H^*$. We use a simple and empirically effective maximum-volume $K$-polytope estimator. Corollary \ref{cor:maxvol} shows its consistency. Other rules of selecting $K$ vertices from the sample hull, such as minimizing the Hausdorff distance of the resulting $K$-polytope to the sample hull, are also expected to yield a consistent estimator. 

\begin{corollary}[Consistency of the maximum-volume $K$-vertex estimator]\label{cor:maxvol}
Under the conditions of Theorem \ref{thm:hull}, let
$\widehat \calH^*_n=\{\widehat h_{n,1}^*,\ldots,\widehat h_{n,K}^*\}\subset S_n$
be the $K$-point subset of $S_n=\conv\{Y_1^*,\ldots,Y_n^*\}$ that corresponds to the $K$-polytope with maximum volume. 
Then \(d_{\mathrm{Haus}}\Big(\widehat \calH^*_n,\ \calH^*\Big)\to_{a.s.} 0\).
\end{corollary}

Corollary \ref{cor:maxvol} gives a consistent estimator $\widehat H^*$ of $H^*$, where the $k$th row of $\widehat H^*$ is $\widehat h^*_{n,k}$. However, as discussed before, $H^*$ may not be scientifically interpretable when the pollutants are measured in incompatible units. Next, we estimate $\Phi$ using this estimate of $H^*$. 

\subsection{Estimation of source attribution matrix} \label{sec:est_phi} 

Since $Y_i^\top = \widetilde W_i^\top H^*$, by Theorem \ref{thm:welldefn}, the source attribution matrix can be written in terms of $H^*$ and $\widetilde W$:
\begin{equation}\label{eq:phiscaled}
\phi_{kj}
=
\frac{\widetilde \mu_k H^*_{kj}}
{\sum_{\ell=1}^K \widetilde \mu_\ell H^*_{\ell j}},
\qquad
\widetilde \mu_k=\mathbb E(\widetilde W_{ik}).
\end{equation}
Define the augmented matrix
$H_{\mathrm{aug}}^*=[\,H^* \ \ 1_K\,]\in\mathbb R^{K\times (J+1)}$ and its right inverse
\ifjasa
$
R(H^*)=H_{\mathrm{aug}}^{*\top}
\bigl(H_{\mathrm{aug}}^*H_{\mathrm{aug}}^{*\top}\bigr)^+
\in\mathbb R^{(J+1)\times K},
$
\else
$$
R(H^*)=H_{\mathrm{aug}}^{*\top}
\bigl(H_{\mathrm{aug}}^*H_{\mathrm{aug}}^{*\top}\bigr)^+
\in\mathbb R^{(J+1)\times K},
$$ 
\fi 
where \((\cdot)^+\) denotes the Moore-Penrose inverse. Recall that $r_i=Y_i^\top1_J$ is the row sum of $Y_i$. Then 
$\mathbb E(r_1)=\mathbb E(Y_1^\top) 1_J =\widetilde\mu^\top H^*1_J = \widetilde\mu^\top 1_K$ due to row-stochasticity of $H^*$, and thus
\ifjasa
$
\bigl[\mathbb{E}(Y_1^\top)\ \ \mathbb{E}(r_1)\bigr]\,R(H^*)
  = \widetilde\mu^\top H^*_{\mathrm{aug}}\,R(H^*)
  = \widetilde\mu^\top.
$
\else
$$
\bigl[\mathbb{E}(Y_1^\top)\ \ \mathbb{E}(r_1)\bigr]\,R(H^*)
  = \widetilde\mu^\top H^*_{\mathrm{aug}}\,R(H^*)
  = \widetilde\mu^\top.
$$ 
\fi 
A natural plug-in estimator of $\widetilde\mu^\top$ is then $\widetilde m_n^\top = \bigl[\,\bar Y_n^\top\ \ \bar r_n\,\bigr]\,R(\widehat H^*)$, where $\bar Y_n = \sum_{i=1}^n Y_i/n$ and $\bar r_n = \sum_{i=1}^n r_i/n$. We work with the augmented matrix because the estimator $\widehat H^*$ may be nearly row-rank deficient in finite samples. Affine independence of $\widehat H^*$, which is easier to guarantee in practice than linear independence, ensures full row rank of $\widehat H_{\mathrm{aug}}^*$. For numerical stability and robustness, we implement our algorithm with the augmented matrix, and we state the theoretical results accordingly.

\begin{corollary}\label{cor:muhat-consistent}
Under the assumptions of Theorem \ref{thm:hull}, if $\widehat H^* \to_{p} H^*$ then $\widetilde m_n \to_{p} \widetilde\mu$. If additionally $\widehat H^* \to_{a.s.} H^*$, then $\widetilde m_n \to_{a.s.} \widetilde\mu$.
\end{corollary}

Combining consistent estimators $\widehat H^*$ and $\widetilde m_n$ then yields a consistent estimator of $\Phi$, up to row permutations. This result extends to the setting with separable multiplicative measurement error (Section \ref{sec:error}), and to partial observation of pollutants. If only a subset $S \subset \{1,\ldots,J\}$ of pollutants is observed with $|S| \geq K$ and the submatrix $H_S$ of full row rank, applying the procedure to $Y_S$ yields a consistent estimator of $\Phi_S$, the submatrix of $\Phi$ with columns restricted to $S$.

The complete procedure, referred to as GeomNMF, is summarized in Algorithm \ref{alg:simple}, while a more detailed algorithm is given in Supplemental Section \ref{sec:algodetails}. The estimation of $H^*$ can also be done using existing endmember extraction (polytope vertex estimation) methods such as N-FINDR, VCA, or XRAY. However, these methods are greedy heuristics that do not guarantee the global optimum. We perform a global exhaustive search over the sample hull vertices, for which we have the consistency guarantee in Corollary \ref{cor:maxvol}. As the size of the minimal set defining the sample hull ($N$) is typically much less than the sample size ($n$), the complexity of the combinatorial search over $N \choose K$ subsets is manageable with modern computing power. 

For very large $N$ (e.g., large $n$ and large $J$), we provide scalable variants that replace the exact convex hull computation with a directional search for extreme points and prune geometrically redundant vertices. An optional greedy volume-maximization step, initialized at the estimated vertices, then refines the solution (Supplementary Material \ref{sec:algodetails}). GeomNMF is implemented in the Python package \texttt{geom-nmf}, available on \href{https://pypi.org/project/geom-nmf/}{PyPI}. Code to reproduce all analyses in this manuscript is available on \href{https://github.com/jinbora0720/GeomNMF}{GitHub}. 

\begin{algorithm}[t]
\caption{\textsc{GeomNMF: Source attribution matrix estimation}}
\label{alg:simple}
\begin{algorithmic}[1]
\Input $Y\in\mathbb{R}_+^{n\times J}$, number of sources $K<J$. 
\Output $\widehat \Phi$
\State \textbf{Row normalization:} set $r \gets Y1_J$ and \ $Y^* \gets \left(\text{diag}(r)\right)^{-1}Y$.
\State \textbf{Sample hull:} compute convex hull $S_n$ of the rows of $Y^*$.
\State \textbf{Max-volume subset:} set $\widehat \calH^*_n = \arg\max_{\{v_1,\dots,v_K\}\subset S_n} \ \mathrm{vol}_{K-1}(\conv\{v_1,\dots,v_K\})$ and let $\widehat H^*$ be the matrix with these vectors as rows. 
\State \textbf{Affine right-inverse:} $R \gets \widehat H_{\mathrm{aug}}^{*T}(\widehat H_{\mathrm{aug}}^{*}\widehat H_{\mathrm{aug}}^{*T})^{+}$.
\State \textbf{Sample mean:} $\widetilde m^\top_n \gets [\overline Y_n^{T}\ \ \bar r_n]R$.
\State \textbf{Source attributions:} $\widehat \Phi \gets$ \eqref{eq:phiscaled} with $(\widetilde\mu, H^*)$ replaced by $(\widetilde m_n, \widehat H^{*})$.
\State \textbf{Return} $\widehat \Phi$.
\end{algorithmic}
\end{algorithm}

\subsection{Estimation of source intensity} \label{sec:w}

In source apportionment, the source intensities $W_i$ are of interest for understanding how emissions vary over time or space. However, due to the scaling indeterminacy, $W_i$ is only identifiable up to source-specific scalings. Since $Y = WH = \widetilde{W}H^*$, it is natural to estimate the scaled intensities $\widetilde{W}$ after recovering $H^*$. The next proposition shows that once $H^*$ is consistently estimated, $\widetilde{W}_i$ can also be consistently recovered for each $i$.

\begin{proposition}\label{prop:wtildeplugin}
Assume the setting of Theorem \ref{thm:hull}. Let $\widehat{H}^*$ be an estimator with $\widehat{H}^* \to_p H^*$. For each observation $Y_i$, let $\widehat{\widetilde{W}}_i^\top = \bigl[\,Y_i^\top\ \ r_i\,\bigr]\,R(\widehat{H}^*)$. Then $\widehat{\widetilde{W}}_i - \widetilde{W}_i \to_p 0 $ for each fixed $i$. If $\widehat{H}^* \to_{a.s.} H^*$, then $\widehat{\widetilde{W}}_i - \widetilde{W}_i \to_{a.s.} 0$ for each fixed $i$.
\end{proposition}

The estimated scaled intensities $\widehat{\widetilde{W}}_i$ enable analysis of source variations across time or space. However, the intensities should be compared only within a source, not across sources, due to the scaling indeterminacy. Finally, Proposition \ref{prop:wtildeplugin} applies to the noiseless case. Without replicate measurements, the intensities generally cannot be recovered under the measurement error.

\subsection{Extension to periodically stationary intensities} \label{sec:asm1stratified}

Strict stationarity of the source intensities in Assumption~\ref{asm:erg} can be restrictive. For example, traffic and heating sources have pronounced diurnal and seasonal patterns. When such patterns are periodic, the assumption can be relaxed to periodic stationarity with ergodicity within each stratum (e.g., hour of the day or month). Under this relaxation, the ergodic arguments in our proofs (Supplementary Material \ref{sec:pf}) hold within each stratum, so $\Phi$ estimated from the full dataset remains consistent, now for the source attribution across all strata, each weighted by its long-run fraction of observations. More details are given in Supplementary Material \ref{subsec:asm1stratified}.

\subsection{Uncertainty quantification and choice of $K$}\label{sec:choose_K}

We quantify estimation uncertainty with $B$ bootstrap resamples. Because factors are identified only up to permutation, each resample's factors are aligned to those from the original data via the row permutation minimizing the total Euclidean distance between corresponding rows of $H^*$.

The results above assume $K$ is known, which is rarely the case in practice. Cross-validation is ill-suited to unsupervised methods such as GeomNMF that require complete data, and reconstruction error decreases monotonically in $K$. Following common practice in source apportionment \citep{bousiotis_study_2022}, we fit a range of candidate values and select $K$ by scientific interpretability together with two bootstrap stability metrics for $\widehat{\Phi}$: (i) the mean cellwise coefficient of variation (CVar), the bootstrap standard error of $\phi_{kj}$ divided by its bootstrap mean, averaged over cells with bootstrap mean at least $c$ (e.g., $c = 0.03$), where lower values indicate greater stability; and (ii) dominant-source rank stability, the bootstrap probability that each element's top source is preserved across resamples, averaged over elements, where higher values indicate greater stability. Because dominant sources typically drive scientific and policy conclusions, we prioritize rank stability when the two metrics disagree.

\section{Simulations}\label{sec:sim}

We assessed our theoretical results through simulation studies reported in Supplementary Material \ref{subsec:sim}, generating $W$ from an iid process (Supplementary Material \ref{sec:sim_iidW}) and from a stationary ergodic process with separable measurement error (Supplementary Material \ref{sec:sim_ar1W}) across a range of sample sizes. To assess robustness beyond the theory, we also introduced measurement error at varying levels of non-separability. In both settings, $\widehat\Phi$ and $\widehat H^*$ converged to their true values as $n$ grew (Figures \ref{fig:sim_iid_nrmse}--\ref{fig:sim_ar1W_heat}). Beyond the scope of our theory, estimation remained robust under mildly non-separable measurement error; under moderately non-separable error, the estimation error of $\widehat\Phi$ did not vanish, as expected, but remained stable as $n$ grew (Figure \ref{fig:sim_ar1W_nfd}). The scalable variants of GeomNMF matched the accuracy of the exact method at a fraction of its runtime and, for large $n$, ran almost as fast as a purely greedy N-FINDR-style search, which iteratively replaces vertices to find a maximum-volume simplex (Figure \ref{fig:sim_iid_runtime}).

\section{Source apportionment of Summit aerosol} \label{sec:GL}

We apply GeomNMF to daily aerosol composition data collected at Summit described in Section \ref{sec:data}, to identify and apportion the natural and anthropogenic sources contributing to Summit aerosol. Guided by the fine/coarse distinction in Section \ref{sec:data}, we select a parsimonious set of tracer elements based on data availability and source representativity: Cu, Ni, Pb, S, and Zn in the fine bin as anthropogenic tracers, and Al, Ca, Cl, Si, and Ti in the coarse bin as natural tracers (bins defined in Section~\ref{sec:GL_prep}).

\subsection{Preprocessing and suitability of the data} \label{sec:GL_prep}

We preprocessed the raw record as follows; full details are given in Supplementary Material \ref{subsec:GL_prep}. First, we retained reported concentrations below the instrumental detection limit as measured, rather than replacing them with a substituted value. The detection limit bounds where a single measurement can be confidently distinguished from background, not the range over which the instrument response tracks concentration, so below-detection values retain real information about relative magnitude and covariation. Substitution collapses the low tail onto a constant, erasing covariation among elements and distorting the data geometry. This choice is consequential because most of the selected elements fall below detection in more than 10\% of samples, and Ni in as many as 40\%. Second, we removed a small number of spuriously high values flagged by their reported measurement uncertainty. Third, to reduce missingness, we aggregated the three smallest size bins into a single `fine' bin [0.09, 0.56) $\upmu\mathrm{m}$ and the three largest into a single `coarse' bin [1.15, 10) $\upmu\mathrm{m}$, which reflect similar formation processes: gas-particle conversion and mechanical grinding followed by aeolian resuspension, respectively. Fourth, we truncated non-positive concentrations, which are physically impossible, to the element-specific minimum positive value, and zero reported uncertainties to element-specific minimum positive uncertainty. The uncertainty adjustment is required only by ESAT \citep{smith_esat_2024}, the Python implementation of PMF used for comparison, which takes per-observation uncertainties as input. GeomNMF and the other comparison methods, the N-FINDR-style greedy search and LS-NMF, require no uncertainty input. Finally, because elemental dynamic ranges differ by orders of magnitude (Figure \ref{fig:GL_raw_timeseries}, Supplementary Material \ref{subsec:GL}), we normalized each element by its range between the minimum and the 95th percentile, which is robust to extreme values. The final dataset comprises 2025 (out of 4780) days of measurements between May 17, 2003 and June 16, 2016, with all calendar months represented (Table \ref{tab:GL_month_coverage}, Supplementary Material \ref{subsec:GL}). The record after mid-2016 remained too sparse in the selected elements to include, even after the bin aggregation.

The Summit DRUM record exercises several features of our framework. The element-specific rescaling above is exactly the class of transformation to which $\Phi$ is invariant (Theorem \ref{thm:welldefn}); GeomNMF consistently estimates the submatrix of $\Phi$ for the 10 selected elements (Section \ref{sec:est_phi}); and temporally dependent source intensities are accommodated by Assumption \ref{asm:erg}. Most consequentially, below-detection observations are pervasive rather than incidental at Summit (Figure \ref{fig:GL_raw_timeseries}), reflecting the genuinely low concentrations characteristic of a remote Arctic site. PMF-type methods require a specified uncertainty for every observation, conventionally handle below-detection values through substitution and uncertainty-inflation rules \citep{polissar_atmospheric_1998, dutton_source_2010}, and downweight or discard low signal-to-noise variables \citep{paatero_discarding_2003}. At Summit, however, the low-concentration observations are not noise to be suppressed but the dominant mode of the data, and the degree of downweighting is not determined by any physical principle.

\subsection{Choice of $K$}

Our main results use $K=6$ sources, selected by the procedure in Section \ref{sec:choose_K} with $B=100$ bootstrap resamples. The $K=6$ solution has higher rank stability, whereas the $K=5$ solution has lower CVar (Table \ref{tab:phi_stability_K}, Supplementary Material \ref{subsec:GL}). Under $K=5$, the anthropogenic tracer Pb is split between two natural mineral sources (Figure \ref{fig:GeomNMF_K5_K6_Phi}, Supplementary Material \ref{subsec:GL}); under $K=6$, it forms a distinct source with Ni, another anthropogenic tracer. Given our priority on rank stability and this clearer separation of anthropogenic from natural contributions, we adopt $K=6$. Smaller $K$ mixes anthropogenic and natural contributions further and is inconsistent with the number of well-established Summit aerosol sources \citep{vancuren_aerosols_2012}. We did not consider $K>6$ given the data dimension ($J=10$).

\subsection{Empirical checks of assumptions}

With the six-source GeomNMF fit, we empirically examine Assumptions \ref{asm:erg}--\ref{asm:sep} on the estimated source intensities. Ergodicity is not directly verifiable from a single realization. Augmented Dickey-Fuller (ADF) and Kwiatkowski-Phillips-Schmidt-Shin (KPSS) tests support stationarity for Sources 2--5 and are inconclusive for Sources 1 and 6 (Table \ref{tab:stationarity}, Supplementary Material \ref{subsec:GL}). Specifically, Sources 1 and 4--6 show seasonal structure in their monthly medians (Figure \ref{fig:seasonality_check}). Our results remain valid under this departure from stationarity, since Assumption \ref{asm:erg} extends to periodic stationarity (Section \ref{sec:asm1stratified}), with $\Phi$ interpreted as an attribution over the sampled record. For probabilistic separability, between 0.1\% (Source 3) and 0.8\% (Source 2) of samples satisfy $\widehat{W}_{ik}^*>0.8$, and, excluding samples selected as vertices, the largest attained share ranges from 0.893 (Source 2) to 0.979 (Source 6), indicating that each source dominates a small fraction of samples.

\begin{figure}[t!]
    \centering
    \includegraphics[width=\textwidth]{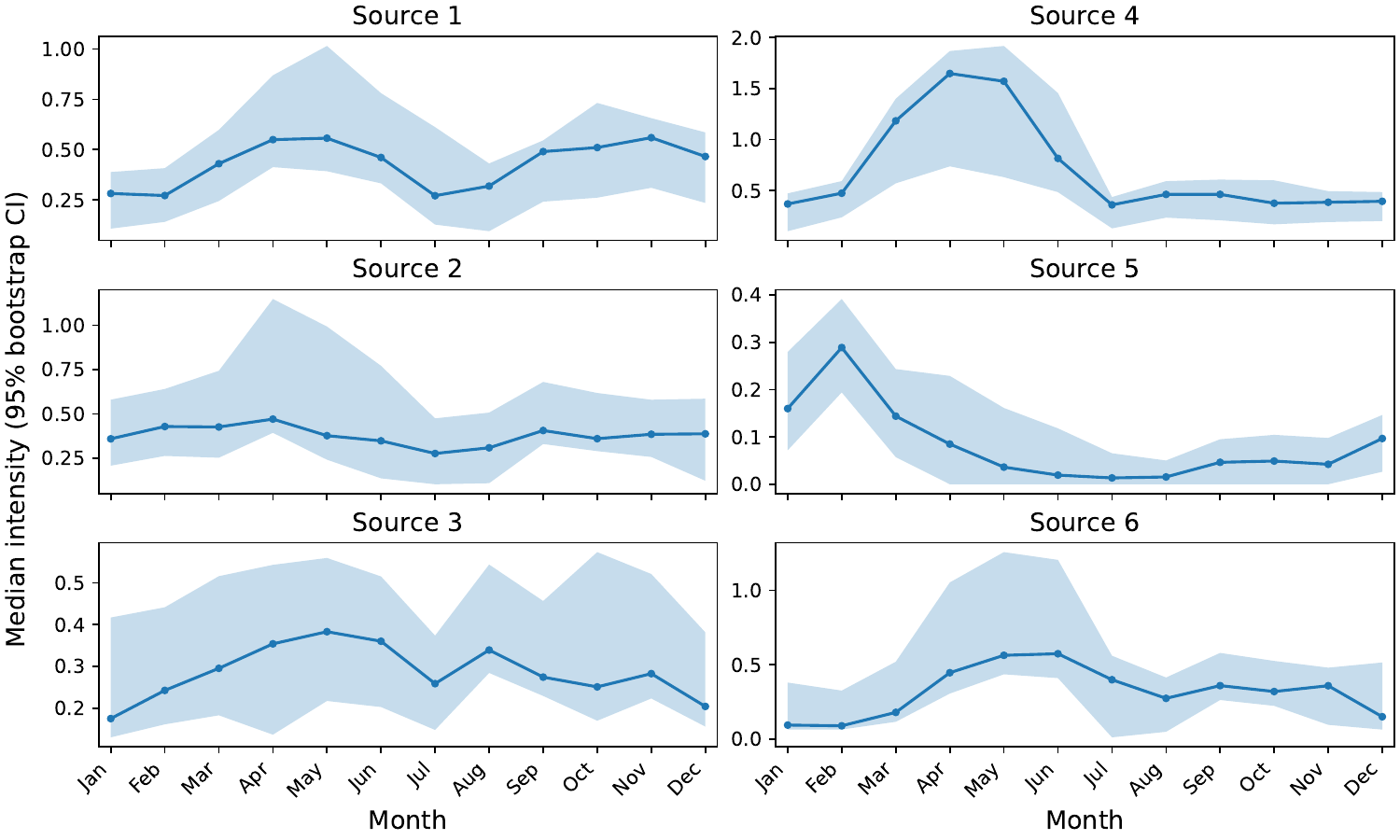}
    \caption{Monthly median of estimated source intensities $\widehat{\widetilde{W}}$. The line shows the median computed from the original data; the band shows a 95\% CI from the bootstrap distribution of the monthly median.}
    \label{fig:seasonality_check}
\end{figure}

\subsection{Results and comparisons with other methods}

We compared GeomNMF against three alternatives: the N-FINDR-style greedy search, LS-NMF, and ESAT. GeomNMF computed the exact convex hull from the data and searched over a pruned set of candidate vertices followed by greedy refinement. LS-NMF was fit with and without regularization. The penalty shrinks the entries of $W$ and $H$ toward zero, yielding a smoother solution. Full parameter specifications are provided in Supplementary Material \ref{subsec:GL_detail}. Figure \ref{fig:GeomNMF_K6_Phi} reports the estimated source attribution matrix $\widehat{\Phi}$ from GeomNMF. The estimated attributions separate the anthropogenic sources (Sources 1--3) from the natural mineral sources (Sources 4--6), as expected. Source 1 accounts for 0.87 of Zn, 0.84 of Cu, 0.13 of Ni, and 0.16 of Al. Source 2 accounts for 0.93 of Pb, 0.23 of Ni, 0.13 of S, and 0.14 of Al. Source 3 accounts for 0.62 of Ni. Source 4 takes 0.92 of Ca, 0.87 of Si, and 0.64 of Al, together with 0.82 of S. Source 5 takes 0.86 of Cl. Source 6 takes 0.84 of Ti. 

The estimated $\widehat{\Phi}$ are broadly similar across methods (Figure \ref{fig:alternative_K6_Phi} in Supplementary Material \ref{subsec:GL}), with notable differences. LS-NMF estimates with and without a penalty differ by far more than their bootstrap uncertainty. The reported precision therefore reflects stability of the regularized optimizer rather than true precision. This is problematic as the penalty weight is typically chosen without a criterion tied to source apportionment. Scientifically, LS-NMF assigns Al to a different source than Ca and Si, although the three are typically grouped as markers of mineral dust \citep{liu_globalscale_2022}. This split more plausibly reflects an artifact of the optimization than a physical distinction. LS-NMF also places fine Pb with coarse Ti on Source 6, and ESAT, which additionally produces markedly larger bootstrap SEs, loads fine Pb and coarse Al on a single source. Neither grouping is physically supported. At Summit, fine Pb and coarse Ti and Al fall into distinct groups by particle size and formation process: fine Pb derives mainly from anthropogenic sources, whereas coarse Ti and Al originate primarily from crustal erosion \citep{hillamo_size_1993}.

Table \ref{tab:comp_methods} reports the mean element-specific rank stability of $\Phi$ and the in-sample RMSE in $Y$. GeomNMF and the greedy search give nearly identical values, as expected given their shared geometric formulation. LS-NMF attains the highest rank stability and lowest RMSE, the latter expected because it directly minimizes in-sample squared error. However, some of its top contributors are inconsistent with established elemental associations, as noted above, and the gap between its regularized and unregularized fits shows that both diagnostics depend on a tuning choice the data alone do not determine. ESAT is worst on both diagnostics. We therefore prefer GeomNMF: its fit is comparable to that of LS-NMF (RMSE 3\% higher than unregularized and 9\% lower than regularized LS-NMF), and its sources are consistent with known aerosol chemistry, separating fine anthropogenic from coarse natural components more clearly, and physically interpretable (Section \ref{sec:hysplit}). 

\begin{figure}[t!]
    \centering
    \includegraphics[width=0.85\textwidth]{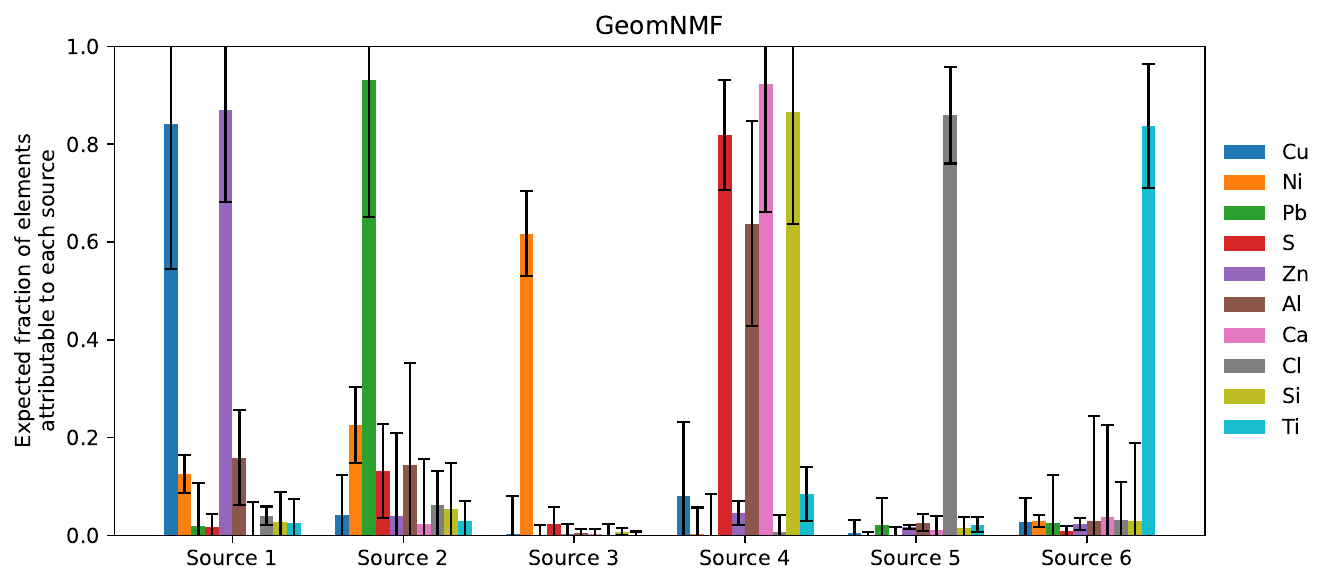}
    \caption{Estimated source attribution matrix $\widehat{\Phi}$ from GeomNMF, with each element summing to one. Bars show the point estimate from the data, and vertical segments show the bootstrap SE.}
    \label{fig:GeomNMF_K6_Phi}
\end{figure}

\begin{table}[htbp]
\centering
\caption{Estimation diagnostics across methods. $R$ is the mean element-specific rank stability of $\Phi$; higher values indicate more consistent identification of the top contributor. RMSE is the in-sample root mean square error in $Y$.} 
\begin{tabular}{lccccc}
\toprule
 & GeomNMF & Greedy search & LS-NMF & Regularized LS-NMF & ESAT \\
\midrule
$R$ & 0.891 & 0.889 & 0.915 & 0.905 & 0.713 \\
RMSE & 0.272 & 0.275 & 0.264 & 0.300 & 0.668 \\
\bottomrule
\end{tabular}
\label{tab:comp_methods}
\end{table}

\subsection{Source interpretation via back-trajectories} \label{sec:hysplit}

To aid source interpretation, we ran the Hybrid Single-Particle Lagrangian Integrated Trajectory model (HYSPLIT; \citealt{stein_noaas_2015}) through the \texttt{splitr} R package (\url{https://github.com/rich-iannone/splitr}) for each day with a DRUM sample. For each sampled day we computed seven-day back trajectories initialized at Summit at 20, 50, 100, and 650 m above ground level and at 00:00, 06:00, 12:00, and 18:00 UTC, driven by the 2.5-degree National Centers for Environmental Prediction and National Center for Atmospheric Research (NCEP/NCAR) reanalysis. Results with the 1-degree GDAS1, available for fewer sampled days, were very similar. Trajectories at 20, 50, and 100 m were nearly identical, reflecting the boundary layer’s influence, whereas the 650 m trajectories, chosen to remain above it, reached farther away. The two sets followed similar source pathways, indicating that the near-surface DRUM samples partly reflect free-tropospheric air. More details on HYSPLITS are given in Supplementary Material \ref{subsec:hysplit}. The reported figures pool trajectories across all release heights and initialization times.

We summarized the HYSPLIT trajectories with the potential source contribution function (PSCF) to identify geographic regions that contribute to elevated pollution at Summit \citep{ashbaugh_residence_1985, cheng_qualitative_1993}. For each grid cell, PSCF estimates the conditional probability that air passing through it arrives at Summit with intensity above a chosen threshold. It is computed as $\text{PSCF}_{ij} = \omega_{ij} \tau^{+}_{ij} / \tau_{ij},$ where $\tau_{ij}$ is the total residence time of all trajectories in cell $(i,j)$ in hours, $\tau^{+}_{ij}$ is the residence time in cell $(i,j)$ contributed by trajectories whose receptor-arrival intensity exceeds the threshold, and $\omega_{ij}$ is the stability weight that downweights sparsely sampled cells \citep{polissar_atmospheric_2001}. A cell with PSCF near one indicates that most air crossing it arrived enriched in the source, flagging it as a probable source region. We used a $100 \times 100$ km grid, the 99th percentile of each source's intensity as the threshold, and stability weight $\omega_{ij} = \min(1,\tau_{ij}/\operatorname{median}(\tau_{ij}))$, where $\operatorname{median}(\tau_{ij})$ is the median residence time across all grid cells. 

\begin{figure}[t!]
    \centering
    \includegraphics[width=\textwidth]{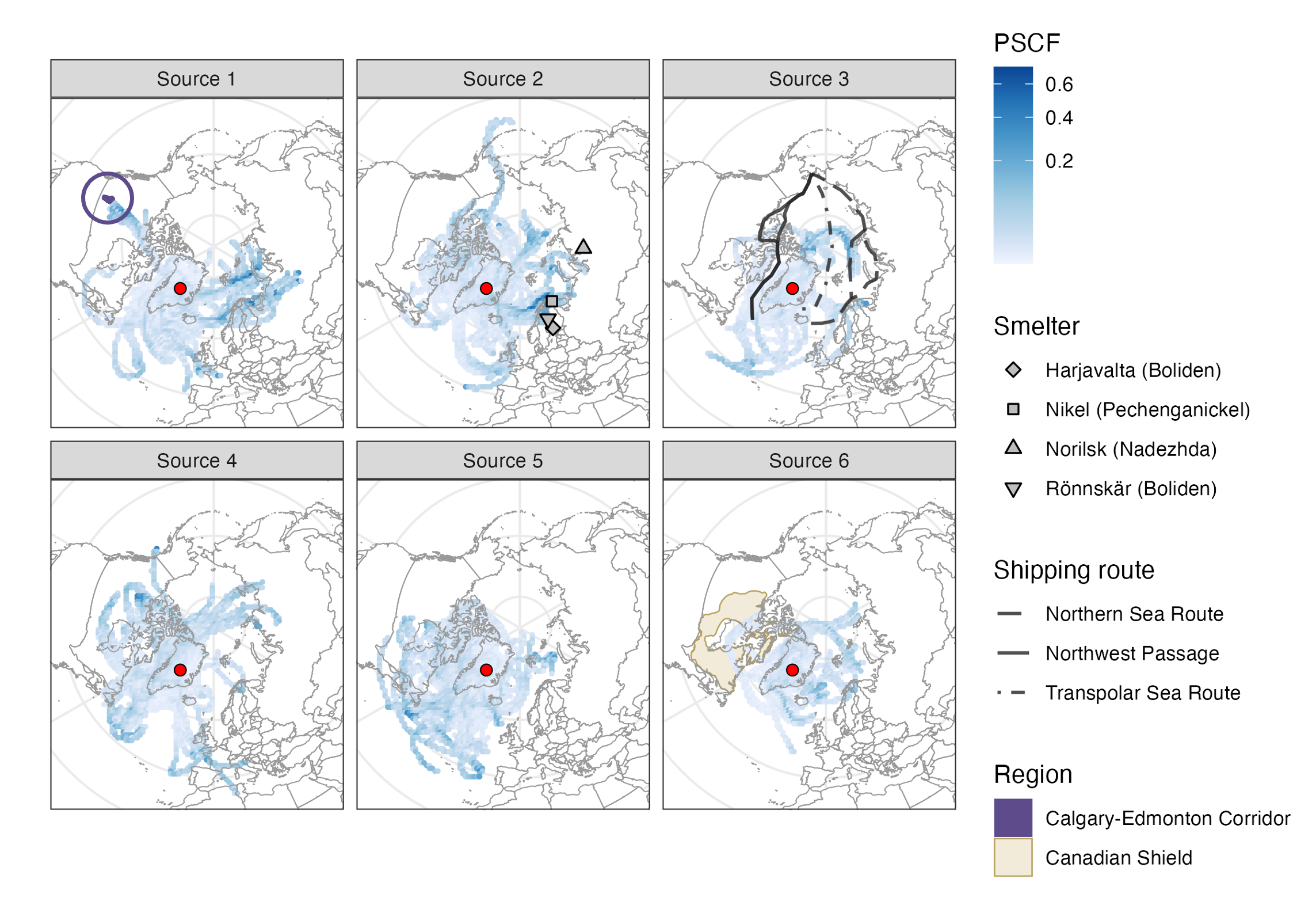}
    \caption{PSCF for six GeomNMF sources at Summit, Greenland, computed from 7-day HYSPLIT back-trajectories driven by NCEP/NCAR reanalysis (2003--2016). For each source, a trajectory day is classified as polluted when its estimated source intensity exceeds the source's 99th percentile. The red dot marks Summit. Selected smelters (symbols), Arctic shipping routes (lines; schematic, based on The Arctic Institute, \url{https://www.thearcticinstitute.org/arctic-maps/}), Calgary-Edmonton Corridor (area; boundary from Statistics Canada), and the Canadian Shield (area; boundary from Natural Resources Canada) are overlaid.
    }
    \label{fig:hysplit_bysource}
\end{figure}

Figure \ref{fig:hysplit_bysource} shows the back trajectories colored by PSCF for each source. Because air masses associated with different sources often pass over the same regions, the trajectories overlap across sources; we therefore focus on the dominant pathways. Source 1 (Zn, Cu, with minor contributions to Ni and Al) has high-PSCF trajectories from lower latitudes in Canada including the Calgary-Edmonton Corridor, northern and western Europe, and Russia. Zn, Cu are associated with vehicle and industrial emissions \citep{andrade_vehicle_2012, pant_estimation_2013}. The Calgary-Edmonton Corridor is one of the most densely populated and urbanized regions of Canada, northern and western Europe are heavily urbanized, and Russia contains major urban-industrial regions. We therefore interpret Source 1 as a mixed urban-industrial factor. Source 2 (Pb, Ni, with minor fractions of S and Al) shows trajectories tracing several known Arctic smelters. Nikel, Norilsk, and Harjavalta are Cu-Ni smelters, and R\"{o}nnsk\"{a}r is a polymetallic (Cu-Pb-Zn) smelter. The co-occurring minor S is consistent with the SO2 these sulfide-ore smelters emit and which oxidizes to sulfate during transport. Because these smelters lie within the same industrialized regions of northern Europe and Russia that Source 1 draws from, Sources 1 and 2 share some high-PSCF source regions. This regional footprint also aligns with Pb-isotope evidence from Svalbard proglacial sediments attributing the anthropogenic atmospheric Pb component to long-range transport from Eastern Europe and Russia \citep{xue_legacy_2025}. Source 2 intensity is also highly persistent, as shown in the short-range autocorrelation function (ACF) in Figure \ref{fig:GeomNMF_W_lag10acf}, Supplementary Material \ref{subsec:GL}, consistent with a continuously operating source with sustained transport such as smelter emissions. ACF values close to one indicate that source intensities are strongly correlated with their recent past. The lag-1 ACF reaches 0.948 and remains high through 10 calendar days or more. Source 3 (Ni) is consistent with a marine vessel factor \citep{celo_chemical_2015, spada_decreasing_2018}. Its elevated PSCF trajectories coincide with Arctic shipping routes, supporting vessel emissions, and shows little contribution from major continental regions. The elevated PSCF over the Atlantic further appears to follow the transatlantic shipping lane linking Europe (e.g., the UK) and North America (e.g., Boston, New York). 

Source 4 (Ca, Si, Al, S) is interpreted as long-transported mineral dust, which could indicate dust that has been aged with sulfate during transport \citep{fitzgerald_comparison_2015}. This long-range character is reflected in a slowly decaying ACF, with a lag-1 value of 0.828. Its monthly pattern in Figure \ref{fig:seasonality_check} peaks in April and May and is otherwise flat. This spring maximum is consistent with previous work: April–May is a peak period for long-range dust transport from distant sources \citep{boo_transport_2023}, and also when high-latitude dust sources in this Arctic sector are most active \citep{bullard_high-latitude_2016}. Forward trajectories in Figure \ref{fig:source4_gobi_sahara} identify the Gobi and Sahara deserts as potential source regions: the Source 4 PSCF and the forward trajectories on the matching 21 days coincide. This is consistent with a previous study at the same site over a different time range \citep{vancuren_aerosols_2012}, which attributed mineral dust at Summit to Asian and African sources. Figure \ref{fig:source4_gobi_supporting} further supports transport from the Gobi across Russia toward the Arctic, shown by dust column mass from the Modern-Era Retrospective Analysis for Research and Applications, Version 2 (MERRA-2). Source 5 (Cl) is identified as sea salt. Its PSCF trajectories concentrate over the North Atlantic. Its monthly pattern, peaking in February and elevated through winter, is consistent with the winter maximum in North Atlantic storm-track intensity \citep{hoskins_annual_2019}. Locally or episodically generated sources such as Sources 3 and 5 show a fast ACF drop, with lag-1 values of 0.177 and 0.279 respectively. 

\begin{figure}[t!]
    \centering
    \includegraphics[height=0.28\textheight]{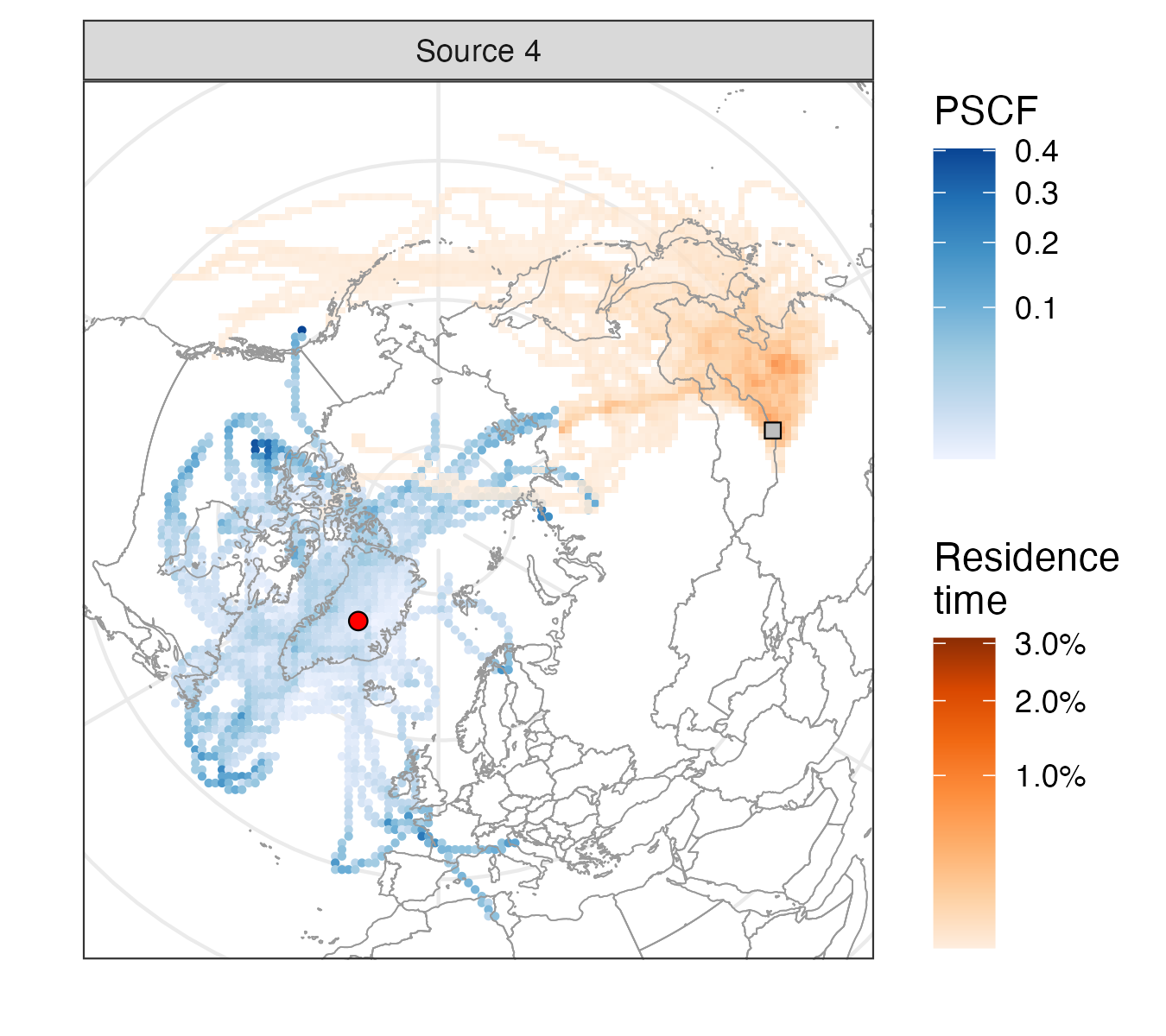}
    \includegraphics[height=0.28\textheight]{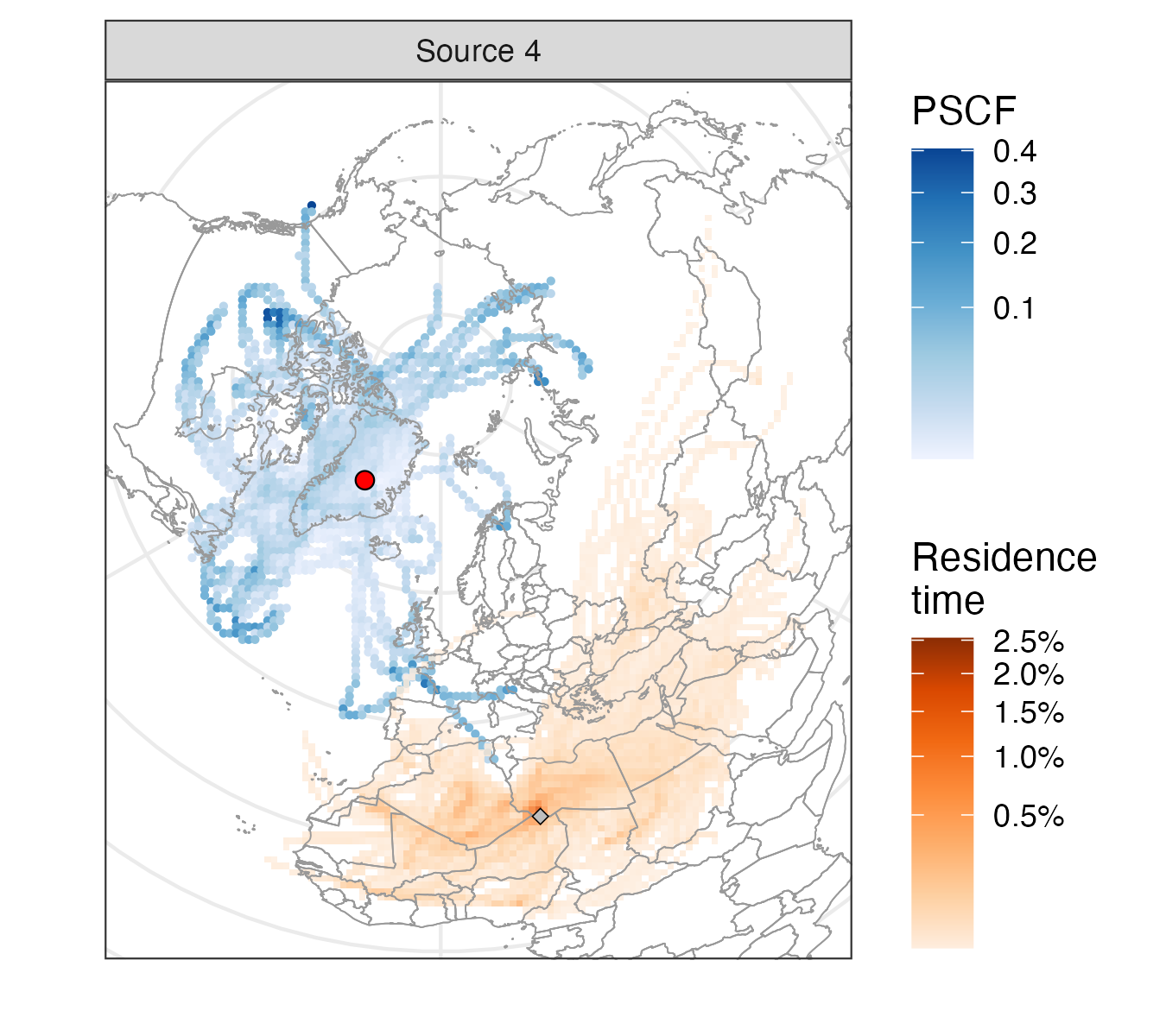}
    \caption{Seven-day forward trajectories released from the Gobi Desert (left, gray square) and the Sahara (right, gray diamond), colored by normalized residence time (\%). Forward trajectories were released on days preceding Summit episodes in which Source 4 intensity exceeds its 99th percentile, offset by the estimated transport time from each source: 14 days for the Gobi and 7 days for the Sahara. Blue points show the Source 4 PSCF field at Summit for reference.}
    \label{fig:source4_gobi_sahara}
\end{figure}

Source 6 (Ti) is identified as a Ti-rich dust factor. Its PSCF trajectories touch the Canadian Shield and encircle the Arctic Ocean. Of these, marine air carries negligible Ti, so the factor's Ti must be crustal in origin. The monthly pattern peaks in May and June, consistent with springtime deflation of exposed surfaces after snowmelt. The Canadian Shield has remarkably thin soil cover, so winds passing over it scour exposed bedrock, glacial till, and dried lake beds. The exposed rock across the eastern and central Shield is dominated by large intrusive anorthositic and gabbroic bodies, which are Ti-bearing \citep{rose_geology_1969}. Wind scouring such Ti-bearing surfaces offers a plausible mechanism for the observed factor. However, this does not mean the Shield influences only Source 6: other sources show elevated PSCF over the same region. For instance, Source 4's trajectories also follow the Canadian Shield pathway, in addition to the better-established Gobi and Saharan origins discussed above. In Source 6, the same Shield-derived Ti becomes the primary signal because no stronger elemental source competes with it. This Ti-rich dust source has not been reported for Summit before. While Asia and North Africa are established dust source regions for Greenland \citep{vancuren_aerosols_2012, ujvari_greenland_2022}, the Canadian Shield as a Ti-rich dust source for Greenland is, to our knowledge, a novel identification.

\section{Discussion} \label{sec:disc}

In this manuscript, we synthesized probabilistic and geometric approaches to show that the source attribution matrix, a population-level estimand for the proportion of each pollutant's concentration attributable to each source, is scale-invariant, statistically identifiable, robust to separable multiplicative measurement error, and consistently estimable by a simple algorithm, GeomNMF, under relaxed assumptions. Beyond the Arctic application, the framework is well suited to newer, unconventional source apportionment settings with large, high-dimensional data, such as low-cost air sensor networks. These networks typically measure a handful of non-speciated pollutants at high spatio-temporal resolution, sometimes in incompatible units (e.g., gases in ppb and PM in $\upmu\mathrm{g\,m^{-3}}$). The scale-invariance of $\Phi$ accommodates such mixed units, and the large sample sizes these networks provide, often millions of observations, are the regime in which our consistency results are most relevant. \citet{jin_source_2025} demonstrated this in Curtis Bay, Baltimore, applying GeomNMF to approximately half a million samples.

Our Arctic aerosol analysis carries several scientific implications. The continuous, multi-year DRUM record allows both short-term and long-term structure in Greenland atmospheric composition to be examined: day-to-day source contributions resolve individual transport episodes, while the multi-year span reveals persistence and seasonality in those contributions. Despite the weak signals inherent to Summit without local emission sources, GeomNMF, complemented by HYSPLIT analysis, recovered interpretable source structures. The resolved attributions are consistent with the broader Arctic aerosol literature, including industrial emissions, smelters, vessels, long-range transport of mineral dusts, and sea salt, yet also specific to Summit, with a crustal signature from the Canadian Shield. By contrast, the competing methods produced groupings that were harder to reconcile with known aerosol chemistry and unreliable uncertainty estimates. 

This application study has some limitations. The geographic source attributions rest on HYSPLIT trajectory and PSCF analysis rather than on the identified estimand, and are therefore interpretive. The reported DRUM data also do not capture all relevant elements, and even where measurements exist, some elements were not utilized because of missingness. The chemistry we can resolve is therefore incomplete. In Source 4, for example, we observed an aging effect involving sulfate. However, nitrate, which has a stronger affinity for mineral dust \citep{sullivan_direct_2007}, cannot be detected by X-ray fluorescence and so is absent from the analysis. Finally, the HYSPLIT analyses do not account for dry and wet deposition, which affects the different particle size bins differently.

Several directions merit further investigation. Extending the framework to handle missing data, outliers, and model misspecification such as more complex forms of non-stationarity and measurement errors are important practical priorities. Looking ahead, Bayesian formulations that enable data-adaptive selection of $K$, propagate uncertainty in $\Phi$ to health and policy metrics, and incorporate priors encoding the conical geometry of NMF problems would be valuable, alongside development of more computationally efficient algorithms.

\section*{Acknowledgement} 
BJ and AD are partially supported by National Institute of Environmental Health Sciences (NIEHS) grant R01ES033739. JC was supported by the Swedish Research Council (grant no. 2026-02282). MDS was supported by NOAA cooperative agreement NA22OAR4320151 and National Science Foundation grant OPP-2137091. During the preparation of this work the authors used chatGPT and Claude for helping with literature review, and improving the clarity of some of the proofs and exposition of the analysis results. After using this tool/service the authors reviewed and edited the content as necessary and take full responsibility for the content of the publication.

\newpage
\begin{center}
\Large{\title{\textbf{Supplementary materials for \\``\papertitle''}}}
\end{center}
\renewcommand\thesection{S\arabic{section}}
\renewcommand\theequation{S\arabic{equation}}
\renewcommand\thelemma{S\arabic{lemma}}

\renewcommand\thefigure{S\arabic{figure}}
\renewcommand\thetable{S\arabic{table}}
\setcounter{figure}{0}
\setcounter{section}{0}
\setcounter{equation}{0}
\setcounter{table}{0}
\setcounter{lemma}{0}


\section{Relevant source apportionment methods} \label{subsec:rev}

\subsection{Algebraic approaches}

Many methods for source apportionment minimize ordinary or weighted least squares $\|Y - WH\|^2$. These include the Chemical Mass Balance (CMB) model \citep{watson_chemical_1991}, which requires known source profiles $H$ and estimates $W$ via precision-weighted least squares. Standard NMF implementations in R \citep{gaujoux_flexible_2010} and Python \citep{pedregosa_scikit-learn_2011} estimate both $W$ and $H$, possibly with regularization to promote sparse solutions. Positive matrix factorization\citep[PMF;][]{paatero_positive_1994}, the most widely used source apportionment algorithm, also minimizes precision-weighted least squares to estimate both $W$ and $H$. 

Identifiability of least-squares-based NMF (LS-NMF) is a well-studied problem in source apportionment. \cite{park_bilinear_1999, sug_park_bilinear_2002} provide multiple sets of sufficient conditions for identifiability of $W$ and $H$ in \eqref{eq:main}. Many of these assumptions impose exact sparsity on the profile matrix $H$. For example, Assumption set $C$ of \cite{sug_park_bilinear_2002} assumes each row of $H$ has at least $K-1$ zeroes, implying that each source does not emit at least $K-1$ of the $J$ pollutants. Assumption set $D$ 
requires a $K \times K$ diagonal submatrix in $H$, meaning each source has at least one unique tracer pollutant in data. These assumptions are rarely valid in practice unless chemically speciated measurements are available with clear tracer species for each source.

Asymptotic theory for LS-NMF is largely underdeveloped. Standard nonlinear least squares theory does not apply directly, as the number of parameters grows with the sample size, and existing results require strong additional assumptions such as identically and independently distributed (iid) rows of $W$ \citep{yang_confirmatory_1994}. Furthermore, established asymptotic results are limited to global optima \citep{park_bilinear_1999, sug_park_bilinear_2002}, which LS-NMF algorithms are not guaranteed to find.

\subsection{Probabilistic approaches}

Probabilistic approaches that assume a stochastic data-generating model are also common in source apportionment. Early work by \citet{bandeen-roche_source_1991} introduced random iid source intensities and established consistency and asymptotic distributional results for an unknown source profile estimator, though in a specialized setup restricted to a single unknown source. More recent approaches offer a probabilistic framework for modeling temporal \citep{park_multivariate_2001, christensen_accounting_2002, heaton_incorporating_2010, rusanen_novel_2023}, spatial \citep{jun_multivariate_2013}, or spatio-temporal \citep{frigeri_bayesian_2025} dependence in environmental data. Estimation proceeds via Bayesian MCMC algorithms \citep{park_multivariate_2001, heaton_incorporating_2010, jun_multivariate_2013, hackstadt_bayesian_2014, rusanen_novel_2023, frigeri_bayesian_2025}, or maximum likelihood estimation with bootstrapping \citep{christensen_accounting_2002}. These probabilistic approaches often rely on strong parametric assumptions such as Gaussian processes with specific covariance classes for spatial dependence, or auto-regressive processes for temporal dependence. While identifiability is achieved through the parametric form and prior choices, to our knowledge, there is little asymptotic theory studying robustness of inference on the NMF factors under misspecification of these parametric assumptions.

\subsection{Geometric approaches} 

If $Y_i$ is generated according to \eqref{eq:main}, each observation satisfies $Y_i = \sum_{k=1}^K W_{ik}h_k$, where $h^{\top}_k$ is the $k$th row of $H$. Geometrically, since $W_{ik}$'s are non-negative, $Y_i$ lies in the cone generated by the rows of $H$ as the corner lines (`extremal rays'), and so does the entire point cloud $\{Y_i\}_{i=1}^n$. Equivalently, the row-normalized data $Y^*_i$ lies in the polytope (simplex) with vertices at the row-normalized rows of $H$. The extremal rays of the cone (or equivalently, the vertices of the simplex), and hence $H$ up to scaling, can thus be identified from the shape of the point cloud, provided that data $\{Y_i\}_{i=1}^n$ offers sufficient geometric coverage of the cone. 

\cite{henry_history_1997}, which is one of the pioneering works drawing the connection between geometry and source apportionment, formalizes the geometric identifiability condition as requiring at least $K-1$ normalized observations $Y_i^*$ lying on each face of the polytope. 
\cite{henry_multivariate_2003} presents the Unmix multivariate receptor model based on this identification idea; however, its implementation \citep{us_epa_epa_2015-1} is not compatible with modern operating systems. \cite{donoho_when_2003} establishes identifiability conditions in terms of the conical geometry, requiring $K-1$ linearly independent data points on each face of the cone, and shows separability together with complete factorial sampling specific in image setting is sufficient to guarantee it. Separability is a classical but strong condition: for each source, there exist pure observations, i.e., times when only that source is active. Geometrically, it requires some $Y_i$'s to lie exactly on the extremal rays of the cone, or equivalently, some $Y_i^*$'s to lie at the vertices of the polytope. End-member extraction algorithms, originally developed in remote sensing literature but later also applied to NMF, exploit separability to find pure observations. These include vertex component analysis \citep[VCA;][]{nascimento_vertex_2005}, which uses random projections and successive identification of extremes in the cone; XRAY \citep{kumar_fast_2013}, which similarly exploits the cone geometry but selects anchor points along directions of largest residual; and N-FINDR \citep{winter_n-findr_1999}, which iteratively replaces vertices to find a maximum-volume simplex. Other conditions relaxing separability include near-separability \citep{gillis_robust_2014, kumar_near-separable_2015}, which requires approximately pure observations, and the sufficiently scattered condition \citep{fu_identifiability_2018}, which requires data to be sufficiently spread inside the cone. 
After estimating $H$, the corresponding weights $W$ can be estimated by solving a non-negative least squares problem \citep{kumar_fast_2013}, among other approaches. 


The deterministic separability or sparsity assumptions underpinning identification in the aforementioned work may not be realized in a finite sample and is unrealistic for environmental applications: it is unlikely that all other sources are simultaneously inactive with only one source active. Furthermore, geometric identifiability results \citep{gillis_fast_2014, arora_computing_2016} address under what conditions a unique factorization can be recovered efficiently, but do not address whether a sample-based estimator converges to the true estimand as the sample size grows. They also tend to ignore the spatio-temporal dependence that environmental data actually exhibit.

Finally, no identification condition in NMF fully resolves the scaling indeterminacy: $W$ and $H$ are determined only up to arbitrary scaling. While a common remedy is to impose an additional sum-to-one constraint on the rows of $H$, the normalized $H$ is sensitive to pollutant-specific unit changes. If pollutants are measured in incompatible units, as is common in air sensor networks where size-resolved particulate matter (PM) is measured in $\upmu\mathrm{g\,m^{-3}}$ and gases in ppb \citep{jin_source_2025}, the row sum is not physically interpretable. Even when units are compatible, normalized profiles can be misleading if a pollutant that dominates a source is not included in the measurement panel, as the remaining pollutants absorb its weight and distort the profile.

\section{Proofs}\label{sec:pf}

Before proving Theorem \ref{thm:welldefn}, we state a standard result for stationary ergodic processes. The proof included here for completeness.

\begin{lemma}\label{lem:erg}
Let \(\{X_t\}_{t \ge 1}\) be a stationary ergodic process on a separable metric space \((\mathsf{X},d)\) with support set \(S\), i.e.,  for every \(a\in S\) and every \(\epsilon>0\) one has \(P\!\left(X_1\in B_\epsilon(a)\right)>0\) for a ball $B_\epsilon(a)$ around $a$. Then for any \(a\in S\) there exists a random subsequence \(\{t_k\}_{k\ge1}\) such that \(X_{t_k}\to a\) almost surely.
\end{lemma}

\begin{proof}
Fix \(a\in S\). For \(n\ge1\) set \(A_n=B_{1/n}(a)\) and define \(I_t^{(n)}=I\{X_t\in A_n\}\), where $I\{\cdot\}$ denotes the indicator function. By stationarity, \(\mathbb E(I_t^{(n)})=P(X_1\in A_n)>0\) for each \(n\). By Birkhoff's ergodic theorem,
\[
\frac{1}{N}\sum_{t=1}^N I_t^{(n)} \to_{a.s.} P(X_1\in A_n)..
\]
Hence for each \(n\) the set \(\{t \in \mathbb N:X_t\in A_n\}\) is infinite with probability one. Construct inductively \(t_1<t_2<\cdots\) by choosing \(t_k\) so that \(X_{t_k}\in A_k\). Then \(d(X_{t_k},a)\le 1/k\) for all \(k\) with probability one, which implies \(X_{t_k}\to_{a.s.}a\).
\end{proof}

\begin{proof}[Proof of Theorem \ref{thm:welldefn}] For notational simplicity we drop tildes and write $\widetilde W_i = Z_i$, $\widetilde H = G$, $\widetilde \mu = \eta$, and $\widetilde \Phi = \Psi$. Let $g_k^\top$ denote the $k$th row of $G$. 

Fix $k \in \{1,\ldots,K\}$. Because $c_k e_k \in$ $\mathbb S_Z$ and $\{Z_i\}_{i\ge 1}$ is stationary and ergodic, Lemma \ref{lem:erg} provides a subsequence $\{Z_{i(m)}\}$ with $Z_{i(m)} \to_{a.s.} c_k e_{k}$ as $m \to \infty$. From $Y_{i(m)}^\top = Z_{i(m)}^\top G$, it follows that $Y_{i(m)}^\top \to_{a.s.} c_k g^\top_k$. Since $Y_i^\top = W_i^\top H$ as well, the sequence $W_{i(m)}^\top H$ converges to $c_k g^\top_k$, so $c_k g^\top_k$, and hence $g^\top_k$, lies in the cone generated by rows of $H$. This cone is closed, 
so $g^\top_k = a_k^\top H$ for some $\alpha_k \in \mathbb R_+^K$. Letting $A$ be the $K\times K$ matrix stacking the row vectors $\alpha_k^\top$ gives
\[
G = A H, \qquad A \ge 0 \mbox{ (entrywise non-negative) }.
\]
By symmetry (interchanging $(W,H)$ and $(Z,G)$), there exists a $K \times K$ matrix $B \ge 0$ with $H = B G.$
Substituting yields  
\[
H = B G = B A H. 
\]
If $H$ has full row rank, then $BA=I_K$, so $A$ is invertible with non-negative inverse $A^{-1}=B$. A non-negative matrix with a non-negative inverse must be monomial \citep{berman_chapter_1979}, i.e. $A = P D$ with a permutation matrix $P$ and a positive diagonal matrix $D$. Therefore, $G = A H = P D H$. After reordering rows (permuting sources) we may take $P=I_K$, giving $G = D H$ and thus $g_k = d_kh_k$. Each row of $G$ coincides with a row of $H$ up to a positive scalar. 

Taking expectations of $Y_i^\top = W_i^\top H = Z_i^\top G = Z_i^\top D H$ gives $\mathbb E(W_i^\top) H = \mathbb E(Z_i^\top) D H$. Since $H$ has full row rank, this implies $\mathbb E(W_i^\top) = \mathbb E(Z_i^\top) D$, so $\mu_k = d_k\eta_k$. Finally,
\[\psi_{kj}=\frac {\eta_k G_{kj}}{\sum_{\ell=1}^K \eta_{\ell} G_{\ell j}} = \frac {(\mu_k/d_k) (d_k H_{kj})}{\sum_{\ell=1}^K (\mu_\ell/d_\ell) (d_\ell H_{\ell j})}  = \phi_{kj},\] establishing $\Psi = \Phi$. 
\end{proof}

\begin{corollary}[Separable log-measurement error preserves Theorem \ref{thm:welldefn}]\label{cor:logerr}
Suppose each data point $Y_i$ admits two representations
\begin{equation} \label{eq:two_rep_logerr}
Y_i^\top
= a_i\, W_i^\top H D_b
= \tilde a_i\, \widetilde W_i^\top \widetilde H D_{\tilde b},
\end{equation}
where at least one of $H,\widetilde H\in\mathbb R_+^{K\times J}$ has full row rank, and the processes $\{W_i\}_{i\ge1}$ and $\{\widetilde W_i\}_{i\ge1}$ satisfy Assumptions \ref{asm:erg}--\ref{asm:sep}. Let $D_b=\text{diag}(b_1,\ldots,b_J)$ and $D_{\tilde b}=\text{diag}(\tilde b_1,\ldots,\tilde b_J)$ have strictly positive diagonal entries. Assume $\{a_i\}_{i\ge 1}$ is a stationary and weakly mixing process with positive values and $\mathbb E(a_i)=1$, independently of $\{W_i\}_{i\ge 1}$, 
and analogously for $\{\tilde a_i\}_{i\ge 1}$. 
Define the reparameterization $W_i':=a_i W_i$ and $H':=H D_b$, and similarly $\widetilde W_i':=\tilde a_i \widetilde W_i$ and $\widetilde H':=\widetilde H D_{\tilde b}$. Let $\Phi'$ denote the source attribution matrix derived from $\mathbb P_{W'}$ and $H'$, and let $\widetilde\Phi'$ derived from $\mathbb P_{\widetilde W'}$ and $\widetilde H'$. Then $\Phi'=\widetilde\Phi'$ up to a permutation of rows. Moreover, $\Phi'=\Phi$, where $\Phi$ is the source attribution matrix from $\mathbb P_{W}$ and $H$.
\end{corollary}

\begin{proof}[Proof of Corollary \ref{cor:logerr}]
Equation \eqref{eq:two_rep_logerr} implies
\[
Y_i^\top=(W_i')^\top H' = (\widetilde W_i')^\top \widetilde H'.
\]
Since $\{a_i\}$ and $\{\tilde a_i\}$ are stationary and weakly mixing, independently of $\{W_i\}$ and of $\{\widetilde W_i\}$, respectively, the product processes $\{W_i'\}$ and $\{\widetilde W_i'\}$ are stationary and ergodic. Stationarity is inherited, and the independent product of a weakly mixing and an ergodic process is ergodic \citep{walters_introduction_1982}. Furthermore, since each $a_i$ or $\tilde a_i$ is a positive scalar, the processes $\{W_i'\}$ and $\{\widetilde W_i'\}$ also inherit probabilistic separability from $\{W_i\}$ and $\{\widetilde W_i\}$. Their finite means satisfy
\[
\mu':=\mathbb E(W_i')=\mathbb E(a_i W_i)=\mathbb E(a_i)\mathbb E(W_i)=\mu,\qquad
\tilde\mu':=\mathbb E(\widetilde W_i')=\mathbb E(\tilde a_i \widetilde W_i)=\mathbb E(\tilde a_i)\mathbb E(\widetilde W_i)=\tilde\mu,
\] 
by the mean and independence assumptions on $\{a_i\}$ and $\{\tilde a_i\}$. Since $D_b$ and $D_{\tilde b}$ are diagonal matrices with strictly positive entries, 
$\mathrm{rank}(H')=\mathrm{rank}(H)$ and $\mathrm{rank}(\widetilde H')=\mathrm{rank}(\widetilde H)$, hence, at least one of $H'$ or $\widetilde H'$ has full row rank.
Applying Theorem~\ref{thm:welldefn} to the two pairs $(W_i', H')$ and $(\widetilde W_i', \widetilde H')$ yields $\Phi'=\widetilde\Phi'$ up to a permutation of rows.

Finally, for each $k,j$,
\[
\phi'_{kj}
=\frac{\mu'_k H'_{kj}}{\sum_{\ell=1}^K \mu'_\ell H'_{\ell j}}
=\frac{\mu_k (H_{kj}b_j)}{\sum_{\ell=1}^K \mu_\ell (H_{\ell j}b_j)}
=\frac{\mu_k H_{kj}}{\sum_{\ell=1}^K \mu_\ell H_{\ell j}}
=\phi_{kj},
\]
and similarly $\widetilde\phi'_{kj}=\widetilde\phi_{kj}$, which proves $\Phi'=\Phi$.
\end{proof}

\begin{proof}[Proof of Theorem \ref{thm:hull}]
Let $\calH^*=\{h_1^*,\dots,h_K^*\}$ and let $\mathbb{S}_{W^*}$ denote the support of $\mathbb{P}_{W^*}$. Since $W_i^*\in\Delta^{K-1}$ and $Y_i^{*\top}=W_i^{*\top}H^*$, $Y_i^* \in \conv(\calH^*)$ and hence
\[
S_n:=\conv\{Y_1^*,\dots,Y_n^*\}\subset\conv(\calH^*).
\]
Since $S_n \subset \conv(\calH^*)$, the Hausdorff distance reduces to the directed distance
\[
d_{\rm Haus}\bigl(S_n,\conv(\calH^*)\bigr)=\sup_{y\in\conv(\calH^*)}\operatorname{dist}(y,S_n).
\]

\textit{Step 1: $\mathbb{D}_{W^*}$ contains each vertex $e_k$.} Fix $\epsilon > 0$. By Assumption~\ref{asm:sep}, for each $k$ there exists $c_k > 0$ such that $\mathbb{P}_W(B(c_k e_k, \epsilon')) > 0$ for any $\epsilon' > 0$ where $B(c_k e_k,\epsilon')$ is the open ball of radius $\epsilon'$ centered at $c_k e_k$. Choose
\[
\epsilon' < \min\left\{\frac{c_k}{2\sqrt{K}},\, \frac{\epsilon c_k}{2(1+\sqrt{K})}\right\}.
\]
For any $w \in \mathbb{R}_+^K$ with $\|w - c_k e_k\| \le \epsilon'$, the Cauchy-Schwarz inequality gives
\[
|1_K^\top w - c_k| = |1_K^\top(w - c_k e_k)| \le \sqrt{K}\epsilon',
\]
so $1_K^\top w \ge c_k - \sqrt{K}\epsilon' \ge c_k/2$, and hence $1/(1_K^\top w) \le 2/c_k$. Setting $w^* = w/(1_K^\top w) \in \Delta^{K-1}$ and applying the triangle inequality,
\begin{align*}
\|w^* - e_k\| &= \left\|\frac{w}{1_K^\top w} - \frac{c_ke_k}{1_K^\top w} + \frac{c_ke_k}{1_K^\top w} - \frac{c_k e_k}{c_k}\right\| \\
    &\le \left\|\frac{w}{1_K^\top w} - \frac{c_ke_k}{1_K^\top w} \right\| + \left\|\left(\frac{c_k}{1_K^\top w} - 1\right)e_k\right\| = \frac{\|w - c_k e_k\|}{1_K^\top w} + \frac{|1_K^\top w - c_k|}{1_K^\top w} \\
    & \le \frac{2\epsilon'}{c_k} + \frac{2\sqrt{K}\epsilon'}{c_k} = \frac{2(1+\sqrt{K})\epsilon'}{c_k} \\
    &\le \epsilon.
\end{align*}
By continuity of the normalization map $w \mapsto w/(1_K^\top w)$ and positivity of $\mathbb{P}_W(B(c_ke_k, \epsilon'))$, it follows that $\mathbb{P}_{W^*}(B(e_k, \epsilon)) > 0$. Since $\epsilon > 0$ was arbitrary, $e_k \in \mathbb{S}_{W^*}$ for every $k$.

\textit{Step 2: The sample convex hull approximates $\conv(\calH^*)$.} Fix $\varepsilon > 0$ and set $\delta = \varepsilon/\|H^*\|_{\mathrm{op}}$, which is well-defined and positive since $H^* = D_h^{-1}H$ has full row rank ($\mathrm{rank}(H) = \mathrm{rank}(H^*)$ because $D_h$ is a diagonal matrix with positive entries). For any $y \in \conv(\calH^*)$, write $y^\top = w^{*\top}H^*$ with $w^* \in \Delta^{K-1}$. Since each $e_k \in \mathbb{S}_{W^*}$, we have $p_k := \mathbb{P}_{W^*}(B(e_k,\delta)) > 0$. By Birkhoff's ergodic theorem,
\[
\frac{1}{n}\sum_{i=1}^n I\{W_i^* \in B(e_k,\delta)\} \to_{a.s.} p_k,
\]
so each ball $B(e_k,\delta)$ is visited infinitely often almost surely. In particular, for each $k$, the index $i_k(n) := \max\{i \le n : W_i^* \in B(e_k,\delta)\}$ is well-defined for all $n \ge N(\varepsilon)$ almost surely, with $W^*_{i_k(n)} \in B(e_k,\delta)$. Define the convex combination
\[
\widehat{y}_n(y) := \sum_{k=1}^K w^*_k Y^*_{i_k(n)} \in S_n.
\]
Since $y^\top = \sum_{k=1}^K w^*_k e_k^\top H^*$ and $Y^{*\top}_{i_k(n)} = W^{*\top}_{i_k(n)}H^*$, for any $n \ge N(\varepsilon)$,
\begin{align*}
\|y - \widehat{y}_n(y)\| &\le \sum_{k=1}^K w^*_k \|e_k^\top H^* - W^{*\top}_{i_k(n)}H^*\| \\
&\le \sum_{k=1}^K w^*_k \|e_k - W^*_{i_k(n)}\|\,\|H^*\|_{\mathrm{op}} \\
&\le \delta\|H^*\|_{\mathrm{op}}\sum_{k=1}^K w^*_k = \delta\|H^*\|_{\mathrm{op}} = \varepsilon.
\end{align*}
Hence $\operatorname{dist}(y, S_n) \le \|y - \widehat{y}_n(y)\| < \varepsilon$ for all $n \ge N(\varepsilon)$, almost surely. 
Since $N(\varepsilon)$ depends only on $\|H^*\|_{\mathrm{op}}$ and $\delta$ through $\varepsilon$, none of which depend on $y$, this bound holds uniformly over $\conv(\calH^*)$. Therefore, for any $\varepsilon > 0$,
\[
\sup_{y\in\conv(\calH^*)}\operatorname{dist}(y,S_n) < \varepsilon \quad \text{almost surely for all } n \ge N(\varepsilon),
\]
and we conclude $d_{\mathrm{Haus}}(S_n,\conv(\calH^*)) \to_{a.s.} 0$.
\end{proof}

\begin{proof}[Proof of Corollary \ref{cor:maxvol}]
For $V=(v_1,\ldots,v_K)^\top\in \R^{K\times J}$, define
\[
G(V)=\operatorname{vol}_{K-1}\big(\conv\{v_1,\ldots,v_K\}\big)=\frac{\sqrt{|\det(AA^\top)|}}{(K-1)!},
\]
where $A=(v_2-v_1, \dots, v_K-v_1)^\top\in \R^{(K-1)\times J}$. Since $A$ is a linear function of the rows of $V$, hence continuous in $V$, and $\det(AA^\top)$ is a continuous function of the entries of $A$, $G$ is continuous on $\R^{K\times J}$. 

For any $v_1,\ldots,v_K \in \conv(\calH^*)$, we have $\conv\{v_1,\ldots,v_K\}\subset \conv(\calH^*)$, so $G(V)\le \operatorname{vol}_{K-1}(\conv(\calH^*))$. Equality holds if and only if $\conv\{v_1,\ldots,v_K\}=\conv(\calH^*)$. Since $H^*$ has full row rank, its rows are affinely independent, and so $\conv(\calH^*)$ has exactly $K$ vertices. So, the only $K$-point subset of $\conv(\calH^*)$ whose convex hull equals $\conv(\calH^*)$ is $\{h_1^*,\ldots,h_K^*\}$ up to permutation. Thus the maximizers of $G$ over $\conv(\calH^*)^K$ are exactly the $K!$ permutations of $\calH^*$.

Since $S_n\to_{a.s.}\conv(\calH^*)$ in Hausdorff distance by Theorem~\ref{thm:hull}, for each $k=1,\dots,K$ there exist points $h_{n,k}\in S_n$ with $\|h_{n,k}-h_k^*\|\to_{a.s.} 0$. By continuity of $G$,
\[
G(h_{n,1},\ldots,h_{n,K})\to_{a.s.} G(h_1^*,\ldots,h_K^*)=\operatorname{vol}_{K-1}(\conv(\calH^*)).
\]
Let $\widehat \calH^*_n=\{\widehat{h}^*_{n,1},\ldots,\widehat{h}^*_{n,K}\}$ be a $K$-point subset of $S_n$ achieving maximum volume. By maximality,
\[
G(\widehat{h}^*_{n,1},\ldots,\widehat{h}^*_{n,K}) \ge G(h_{n,1},\ldots,h_{n,K}),
\]
and since $G$ cannot exceed $\operatorname{vol}_{K-1}(\conv(\calH^*))$, we conclude
\[
G(\widehat{h}^*_{n,1},\ldots,\widehat{h}^*_{n,K})\to_{a.s.} \operatorname{vol}_{K-1}(\conv(\calH^*)).
\]

It remains to show this implies $d_{\mathrm{Haus}}(\widehat \calH^*_n, \calH^*)\to_{a.s.} 0$. Since all points $\widehat{h}^*_{n,k}\in S_n\subset \conv(\calH^*)$ lie in a fixed compact set $\conv(\calH^*)$, every subsequence of $\{(\widehat{h}^*_{n,1},\ldots,\widehat{h}^*_{n,K})\}$ has a further subsequence converging to some $(v_1,\ldots,v_K)\in\conv(\calH^*)^K$. By continuity of $G$, the limit satisfies $G(v_1,\ldots,v_K)=\operatorname{vol}_{K-1}(\conv(\calH^*))$, so $\{v_1,\ldots,v_K\}=\{h_1^*,\ldots,h_K^*\}$ by the maximizer characterization above. Since every subsequence has a further subsequence converging to $\calH^*$, the full sequence satisfies
\[
d_{\mathrm{Haus}}(\widehat \calH^*_n,\ \calH^*)\to_{a.s.} 0. \qedhere
\] 
\end{proof}

\begin{proof}[Proof of Corollary \ref{cor:muhat-consistent}]
By stationarity and ergodicity, $(\bar Y_n, \bar r_n) \to_{a.s.} (\mathbb{E}(Y_1), \mathbb{E}(r_1))$, and thus $(\bar Y_n, \bar r_n) \to_{p} (\mathbb{E}(Y_1), \mathbb{E}(r_1))$. Note that matrix multiplication is continuous. Since $H^*$ has full row rank, $H^*_{\mathrm{aug}}$ has also full row rank, and the Moore-Penrose inverse is continuous in a neighborhood of $H^*_{\mathrm{aug}}$. Therefore, by the continuous mapping theorem, $\widehat H^* \to_p H^*$ implies $R(\widehat H^*) \to_p R(H^*)$, which in turn gives
\[
  \widetilde m_n^\top = \bigl[\,\bar Y_n^\top\ \ \bar r_n\,\bigr]\,R(\widehat H^*)
  \;\to_p\;
  \bigl[\,\mathbb{E}(Y_1^\top)\ \ \mathbb{E}(r_1)\,\bigr]\,R(H^*) = \widetilde\mu^\top.
\]
The almost-sure version follows by replacing $\widehat H^* \to_p H^*$ with 
$\widehat H^* \to_{a.s.} H^*$ throughout.
\end{proof}


\begin{proof}[Proof of Proposition \ref{prop:wtildeplugin}]
Fix $i$. Since $Y_i^\top = \widetilde{W}_i^\top H^*$ and each row of $H^*$ sums to one, 
\[
r_i = Y_i^\top 1_J = \widetilde{W}_i^\top H^* 1_J = \widetilde{W}_i^\top 1_K,
\]
and hence
\[
\bigl[\,Y_i^\top\ \ r_i\,\bigr] = \widetilde{W}_i^\top H^*_{\mathrm{aug}}.
\]
Since $H^*_{\mathrm{aug}}$ has full row rank, $H^*_{\mathrm{aug}} R(H^*) = I_K$, and it follows that
\[
\bigl[\,Y_i^\top\ \ r_i\,\bigr] R(H^*) = \widetilde{W}_i^\top H^*_{\mathrm{aug}} R(H^*) = \widetilde{W}_i^\top.
\]
Define $\widehat{\widetilde{W}}_i^\top = \bigl[\,Y_i^\top\ \ r_i\,\bigr] R(\widehat{H}^*)$. As shown in the proof of Corollary \ref{cor:muhat-consistent} shows that $\widehat H^* \to_p H^*$ implies $R(\widehat H^*) \to_p R(H^*)$. Since $\bigl[\,Y_i^\top\ \ r_i\,\bigr]$ is fixed for the given $i$ and matrix multiplication is continuous, the continuous mapping theorem gives
\[
\widehat{\widetilde{W}}_i^\top = \bigl[\,Y_i^\top\ \ r_i\,\bigr] R(\widehat{H}^*) \to_p \bigl[\,Y_i^\top\ \ r_i\,\bigr] R(H^*) = \widetilde{W}_i^\top.
\]
The almost-sure version follows by replacing $\to_p$ with $\to_{a.s.}$ throughout.
\end{proof}

\subsection{Extension to periodically stationary intensities}\label{subsec:asm1stratified}

We briefly discuss how our theoretical results extend to periodically stationary processes. Assumption \ref{asm:erg} enters the proofs in three ways: through Lemma \ref{lem:erg}, which guarantees that the intensity sequence visits every neighborhood of each point of its support; through the ergodicity of the product processes in Corollary \ref{cor:logerr}; and through the law of large numbers in Corollary \ref{cor:muhat-consistent}. All three extend to intensities that are periodically stationary with a fixed period $p$ (e.g., $p = 7$ for daily data with weekly seasonality) and, in addition, ergodic within each stratum (e.g., all Mondays). Lemma \ref{lem:erg} then simply applies within each stratum. 
Weak mixing is preserved under subsampling, so the product argument in Corollary \ref{cor:logerr} applies to each stratum subsequence unchanged. The sample averages in Corollary \ref{cor:muhat-consistent} converge to averages under the mixture of the stratum-specific distributions, weighted by the long-run fraction of observations falling in each stratum. Accordingly, $\Phi$ is interpreted as the source attribution under this mixture. Under uniform sampling, the weights are equal across strata; when sampling coverage is uneven across the calendar, they shift toward better-sampled strata.

\newpage
\section{Detailed algorithm}\label{sec:algodetails}

For very large $N$ (e.g., large $n$ and large $J$), we suggest two heuristics to improve scalability. First, exact convex hull computation may be replaced by a directional search for extreme points, in which we sample many random directions and retain observations that are most extreme along those directions and their opposites, following the idea of the pixel purity index approach \citep{boardman_mapping_1995}. Second, the candidate vertices on the sample hull may be pruned by removing geometrically redundant vertices, since many empirical extremes are nearly indistinguishable. These steps substantially improve scalability while retaining the essential geometry of the search space. However, when either heuristic is used, the solution may no longer be globally optimal over the original sample hull vertices. In this case, we recommend an optional refinement step: a greedy volume-maximization search initialized at the estimated vertices, which performs local swaps to recover extreme points that the directional search or pruning may have missed. We present the detailed GeomNMF algorithm here, including scalability heuristics. 

\begin{algorithm}[h!]
\caption{\textsc{Detailed} GeomNMF \textsc{for source attribution matrix estimation}}
\label{alg:sourcexray}
\begin{algorithmic}[1]
\Input $Y\in\mathbb{R}_+^{n\times J}$, number of sources $K < J$, optional random direction search flag, optional pruning flag, optional greedy refinement flag (requires at least one of: direction search, pruning), tolerances $\varepsilon$.
\Output $\widehat \Phi$
\State \textbf{Row normalization:} set $r \gets Y1_J$ and \ $Y^* \gets \left(\mathrm{diag}(r)\right)^{-1}Y$.
\State \textbf{Intrinsic projection:} set $Y_{\mathrm{red}}\!\gets Y^*[:,1{:}J-1]$; compute the column-means $m_{\mathrm{red}}\!\gets \tfrac{1}{n}1_n^\top Y_{\mathrm{red}}$; center $Y_c\!\gets Y_{\mathrm{red}}-1_n m_{\mathrm{red}}^\top$; compute thin SVD $Y_c=U\Sigma V^\top$ and take $B\!\gets V_{[:,1:r_B]}$ with rank $r_B$; finally set $Z\!\gets Y_c\,B$.
\State \textbf{Hull vertices:} compute convex hull of $Z$; let $\mathcal I_\mathrm{hull}$ be its vertex indices; set $H^*_\mathrm{cand} \gets Y^*[\mathcal I_\mathrm{hull},:]$. \textit{*\textbf{Option} (Directional extremes)}: scale $Z$ to $Z_{\mathrm{w}}$ to have identity covariance. Sample $T$ random unit directions $\{u_t\}_{t=1}^T$; for each $u_t$, select the top $k$ observations with largest values of $Z_{\mathrm{w}}u_t$ and the top $k$ with smallest values. Let $\mathcal I_{\mathrm{dir}}$ be the resulting unique candidate indices, ordered by how often they are selected. Set $H^*_{\mathrm{cand}} \gets Y^*[\mathcal I_{\mathrm{dir}}, :]$.
\State \textit{*\textbf{Option} (Pruning)}: cluster $H^*_\mathrm{cand}$ (e.g., mini-batch $k$-means); keep one representative per cluster to obtain pruned $H^*_\mathrm{cand}$.
\State \textbf{Max-volume selection (global search):} for each $K$-subset $V\subset H^*_{\mathrm{cand}}$, compute its intrinsic $(K-1)$-volume $\log G(V)$; choose $V^*=\arg\max_{V}\log G(V)$ and set $\widehat H^*\gets V^*$. 
\State \textit{*\textbf{Option} (Greedy refinement)}: \textit{only available if \textbf{Option} (Directional extremes) or \textbf{Option} (Pruning) is active}. Approximate via an N-FINDR-style greedy search initialized at $V^*$, replacing $V^*$ with the locally optimal set found and $\widehat H^*\gets V^*$.
\State \textbf{Affine right-inverse:} augment $\widehat H^*_\mathrm{aug}\gets[\widehat H^*\ \ 1_K]$, $Y^*_\mathrm{aug}\gets[Y^*\ \ 1_n]$; compute $R(\widehat H^*) \gets \mathrm{pinv}(\widehat H^*_\mathrm{aug})$ and $W^*_\mathrm{raw}\gets Y^*_\mathrm{aug}R(\widehat H^*)$. 
\State \textbf{Simplex projection:} replace entries of $W^*_\mathrm{raw}$ by $\max(W^*_\mathrm{raw}, \varepsilon)$ for small $\varepsilon>0$, then normalize each row to sum $1$; denote result by $\widehat W^*$.
\State \textbf{Sample mean:} set $\widehat{\widetilde W} \gets \mathrm{diag}(r)\widehat W^*$ and compute $\widetilde m_n$ as the column means of $\widehat{\widetilde W}$.
\State \textbf{Source attributions:} set
\[
\widehat \phi_{kj} \gets \frac{\widetilde m_{n,k}\widehat H^*_{k j}}{\sum_{\ell=1}^{K}\widetilde m_{n,\ell}\widehat H^*_{\ell j}},\quad k=1,\dots,K, \ j=1,\dots,J.
\]
\State \textbf{Return} $\widehat \Phi$.
\end{algorithmic}
\end{algorithm}

\newpage 
\section{Simulations} \label{subsec:sim}

\subsection{Independent and identically distributed $W$}
\label{sec:sim_iidW}

We simulate an $n\times J$ non-negative data matrix $Y$ as $Y=WH$, with $W\in\R_+^{n\times K}$ and $H\in\R_+^{K\times J}$. We set $J=10$ and $K=5$. We construct $H$ as follows. A large set of candidate vectors in $\mathbb{R}^J$ is sampled with iid entries from an $\mathrm{Exp}(1)$ distribution. The convex hull of the candidate set is then computed, and $K$ of its vertices are selected to form the rows of $H$. For this set of simulations, we assume $W$ is iid across rows. Independently for each column $k\in\{1,\dots,K\}$, we specify a finite mixture on $\log W_{\cdot k}$:
$$
\log W_{ik} \mid Z_{ik}=c \sim N(\mu_{kc}, \sigma_{kc}^2), 
\qquad \mbox{pr}(Z_{ik}=c)=\pi_{kc},
$$
with mixture components $C_k \sim \mathrm{Pois}(3)+1$, mixture weights $\pi_{k\cdot} \sim \mathrm{Dirichlet}(1_{C_k})$, and component parameters $\mu_{kc} \sim \mathrm{Unif}(-1,1)$ and $\sigma_{kc}\sim \mathrm{Unif}(0.1,1)$. This yields strictly positive $W_{ik}=\exp(\log W_{ik})$. The population mean for column $k$ is available in closed form,
$$
\mu_k = \mathbb E(W_{ik}) = \sum_{c=1}^{C_k} \pi_{kc}\, \exp\!\big(\mu_{kc} + \tfrac12\sigma_{kc}^2\big).
$$ We vary $n \in \{100, 300, 1500, 10000, 100000, 500000\}$ and, for each $n$, repeat the data generation and estimation 50 times. 

Let $\widehat{\Phi} = (\widehat{\phi}_{kj})$ denote an estimate for $\Phi$ after permuting rows to minimize the total Euclidean distance between the true and estimated rows. We use normalized root mean squared error (NRMSE), and normalized Frobenius distance (NFD) as performance measures, defined in Table \ref{tab:error-metrics}. 

\begin{table}[htbp]
\centering
\caption{Performance measures for source attribution matrix estimates.}
\vspace{5pt}
\begin{tabular}{cc}
\toprule
Metric & Formula \\ \midrule
NRMSE & $\frac{1}{K}\sum_{k=1}^K \left(\frac{\sqrt{\frac{1}{J}\sum_{j=1}^{J}\bigl(\phi_{k j}-\widehat{\phi}_{k j}\bigr)^{2}}}{\lVert \phi_{k\cdot}\rVert_2}\right)$ \\ 
NFD & $\lVert \Phi-\widehat{\Phi}\rVert_F/\lVert \Phi \rVert_F$ \\ \bottomrule
\end{tabular}
\label{tab:error-metrics}
\end{table}

As $n$, $J$, and $K$ grow, exact exhaustive search becomes computationally intractable. To evaluate the trade-off between scalability and accuracy, we compared five algorithmic variants: (1) exact exhaustive search, (2) exhaustive search over up to $200K$ directional extremes along 20000 randomly generated directions, (3-4) two pruned variants of algorithm (2) retaining the $10K$ and $15K$ candidates respectively, followed by greedy refinement, and (5) a purely greedy search (N-FINDR-style) initialized via the Automatic Target Generation Process (ATGP) \citep{plaza_impact_2006}. Algorithms (1)--(2) were feasible only for $n \in \{100, 300, 1500\}$.

\begin{figure}
    \centering
    \includegraphics[width=0.9\linewidth]{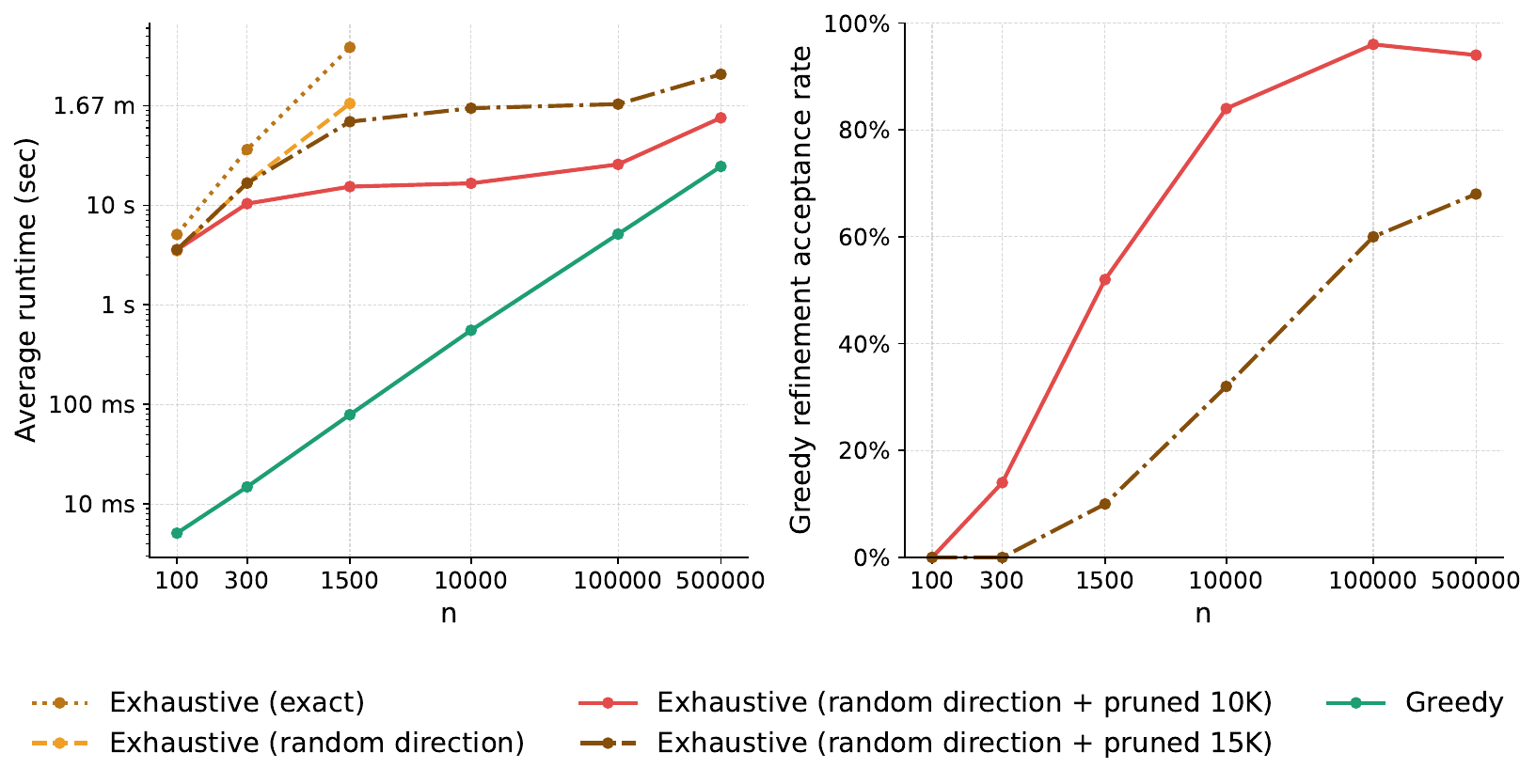}
    \caption{Average runtime per replicate in log scale (left; ms = milliseconds, s = seconds, m = minutes) and greedy refinement acceptance rate (right) across algorithmic variants and samples sizes $n$.}
    \label{fig:sim_iid_runtime}
\end{figure}

Figure \ref{fig:sim_iid_runtime} summarizes runtime and greedy refinement acceptance rates across sample sizes. The left panel shows that the greedy search scales log-linearly in $n$, remaining the fastest option throughout. Exhaustive search variants, while expensive at small $n$, grow much more slowly once pruning and random-direction initialization are applied, making them viable at larger scales. The right panel presents the greedy refinement acceptance rate, measuring how often greedy steps improve upon the exhaustive solution's volume. The acceptance rate increases steadily with $n$, and is substantially higher under the more aggressive $10K$ pruning than under $15K$. Crucially, at small $n$ where exact search was feasible, all faster variants achieved nearly identical accuracy to the exact solution, and at large $n$ the two pruned variants performed indistinguishably from each other (not shown). Together, these results indicate that aggressive pruning can substantially reduce runtime with minimal loss in estimation accuracy.

Because all algorithmic variants yielded closely comparable accuracy, subsequent results are reported for the exhaustive search with random-direction initialization pruned to $10K$. Figure \ref{fig:sim_iid_nrmse} shows the NRMSE of $\widehat{H}^*$ and $\widehat{\Phi}$ across 50 replicates for each value of $n$ (NFD results, which lead to the same conclusions, are omitted). Both NRMSEs decrease monotonically with $n$, and the rate of improvement is slightly faster for $\widehat{\Phi}$ than for $\widehat{H}^*$, reflecting the additional benefit of improved estimation of $\widetilde{W}$ at larger sample sizes. These findings are consistent with the theory of asymptotic convergence of $\widehat{H}^*$, $\widetilde{W}$, and $\widehat{\Phi}$ in the iid $W$ setting.

\begin{figure}
    \centering
    \includegraphics[width=0.65\linewidth]{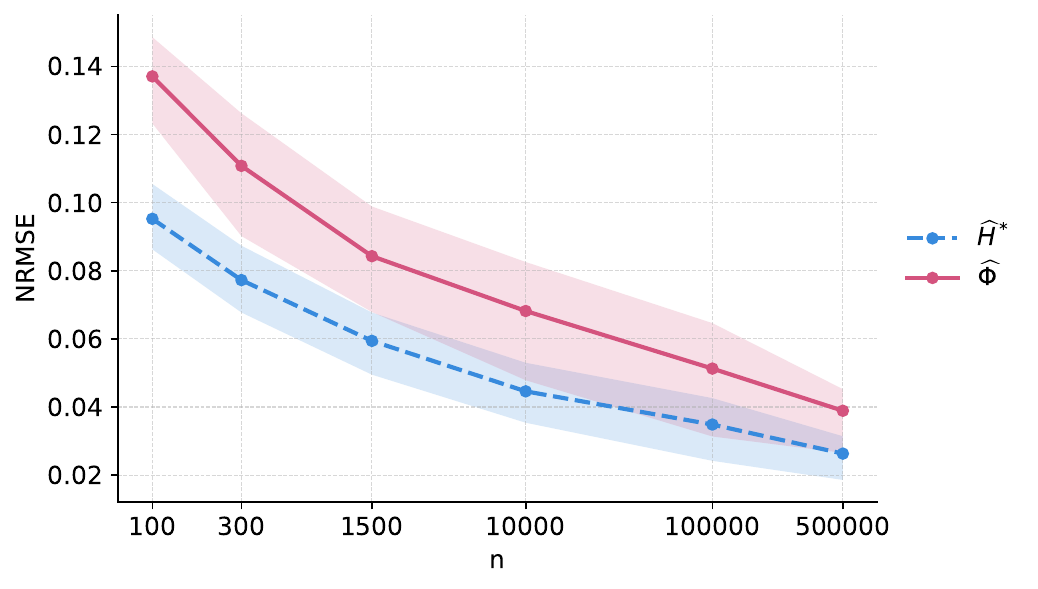}
    \caption{NRMSE of $\widehat H^*$ (blue dashed) and $\widehat \Phi$ (pink solid) as a function of $n$ based on 50 replicates in the iid $W$ setting, obtained via exhaustive search with random-direction initialization pruned to $10K$. Lines show the average and shaded bands the interquartile range.}
    \label{fig:sim_iid_nrmse}
\end{figure}

Figure \ref{fig:sim_iid_scatter} presents scatter plots of all $K \times J = 50$ entries of $\widehat{\Phi}$ against their true values, pooled over 50 replicates, for each sample size. At small $n$ the estimates are widely dispersed around the 45-degree line, but the scatter concentrates tightly along it as $n$ increases. A mild regression-to-the-mean bias (slopes below one) is visible particularly at small $n$; this is expected given that $\Phi$ lies in $[0,1]$ with columns constrained to sum to one.

\begin{figure}[!b]
    \centering
    \includegraphics[width=0.98\linewidth]{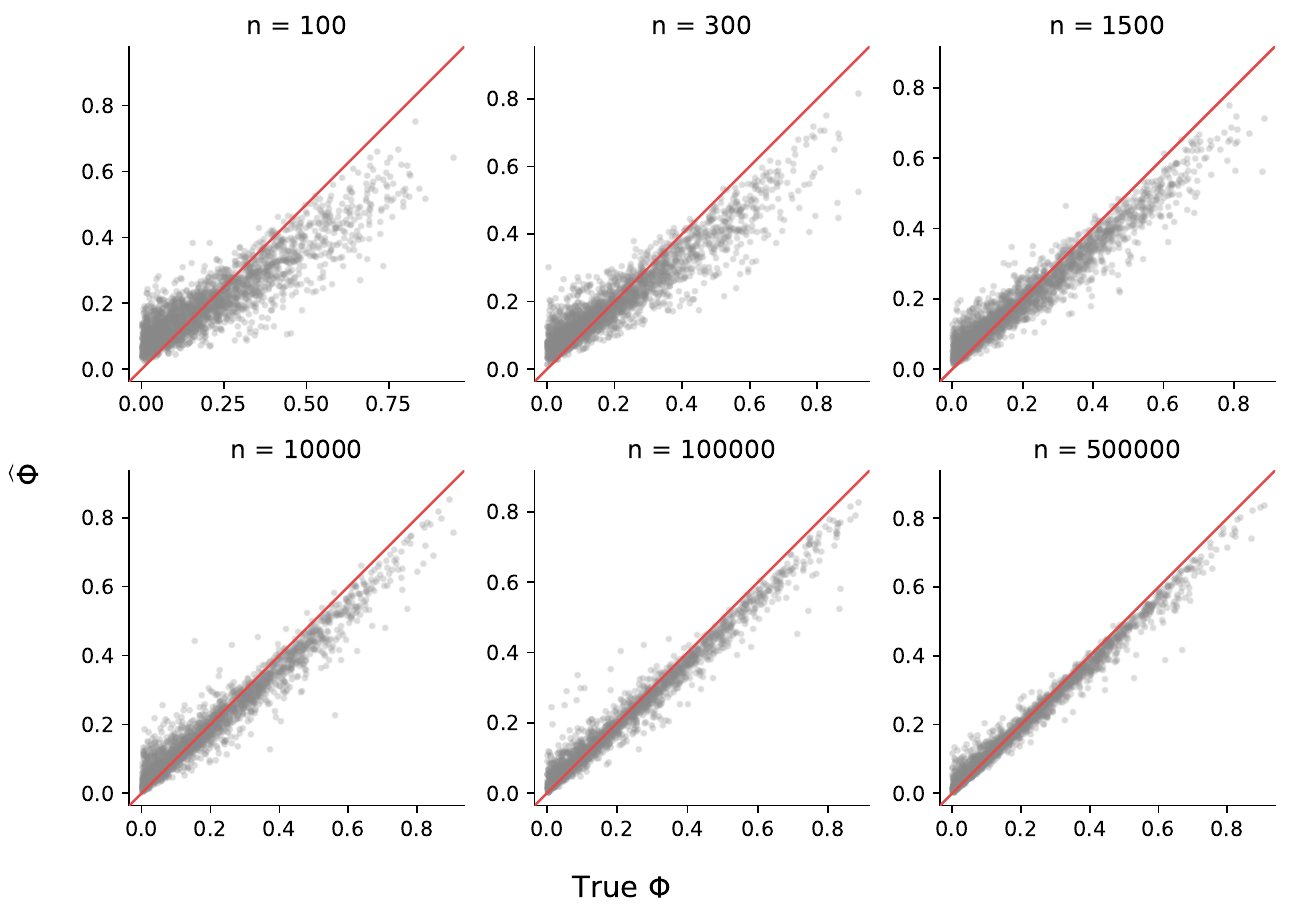}
    \caption{Scatter plots of $K\times J=50$ elements of estimated versus true $\Phi$, pooled over 50 replicates in the iid $W$ setting (exhaustive search, random-direction initialization, pruned to $10K$). The red line denotes the 45-degree reference.}
    \label{fig:sim_iid_scatter}
\end{figure}

Figure \ref{fig:sim_iidW_H} displays $\widehat{H}^*$ alongside the true $H^*$ overlaid on a pairwise scatter plot of a randomly selected pair of variables from $Y^*$, for randomly selected replicates at $n \in \{300, 100000, 500000\}$. As $n$ increases, $\widehat{H}^*$ converges visually to $H^*$. Some true vertices appear interior to the 2D projection despite being extreme in the full $J$-dimensional space. The estimator operates on the full-dimensional cloud and correctly recovers these vertices nonetheless.

\begin{figure}
    \centering
    \includegraphics[width=0.32\linewidth]{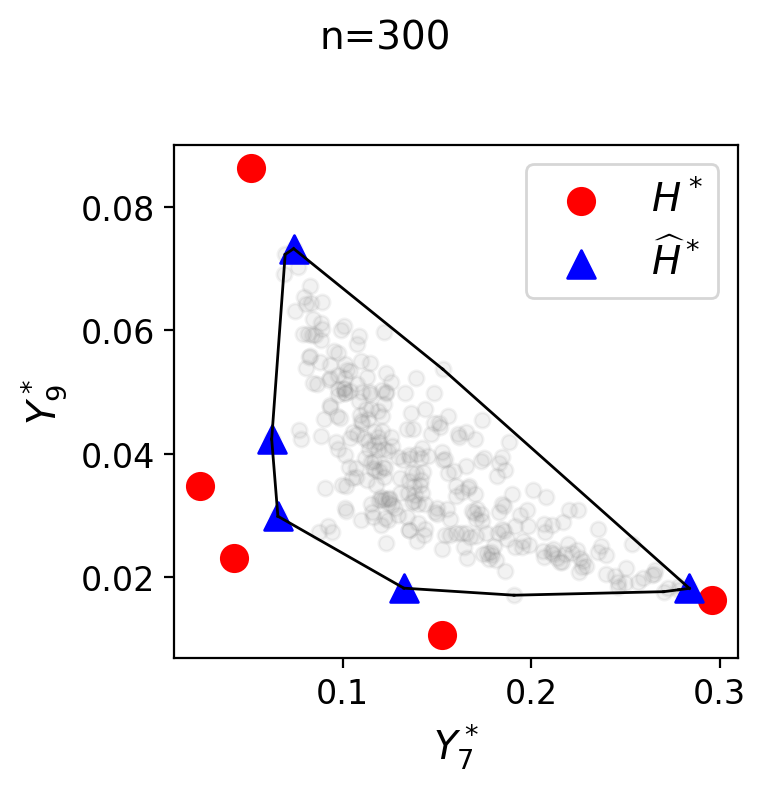}
    \includegraphics[width=0.32\linewidth]{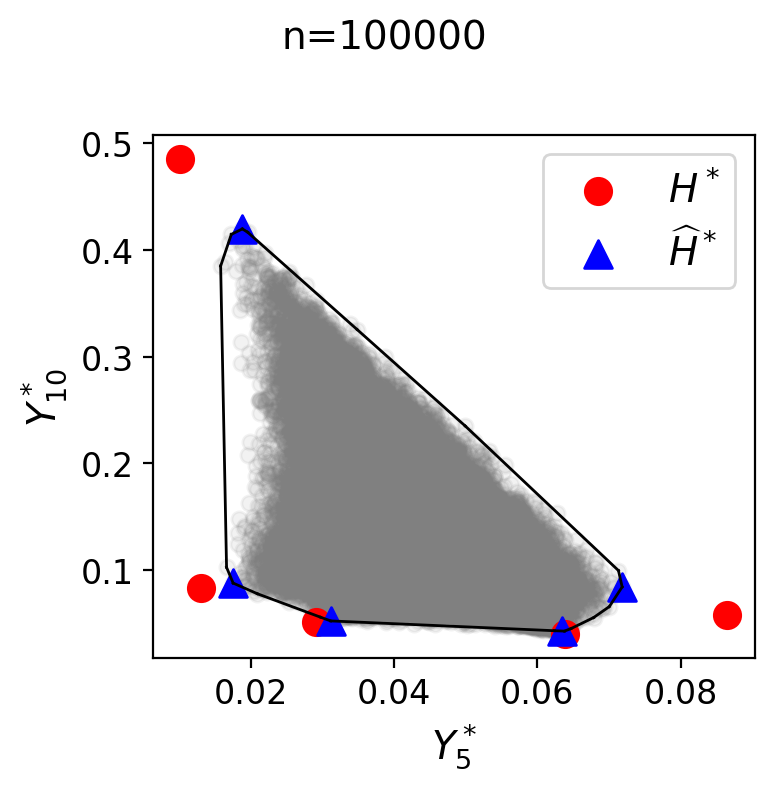}
    \includegraphics[width=0.32\linewidth]{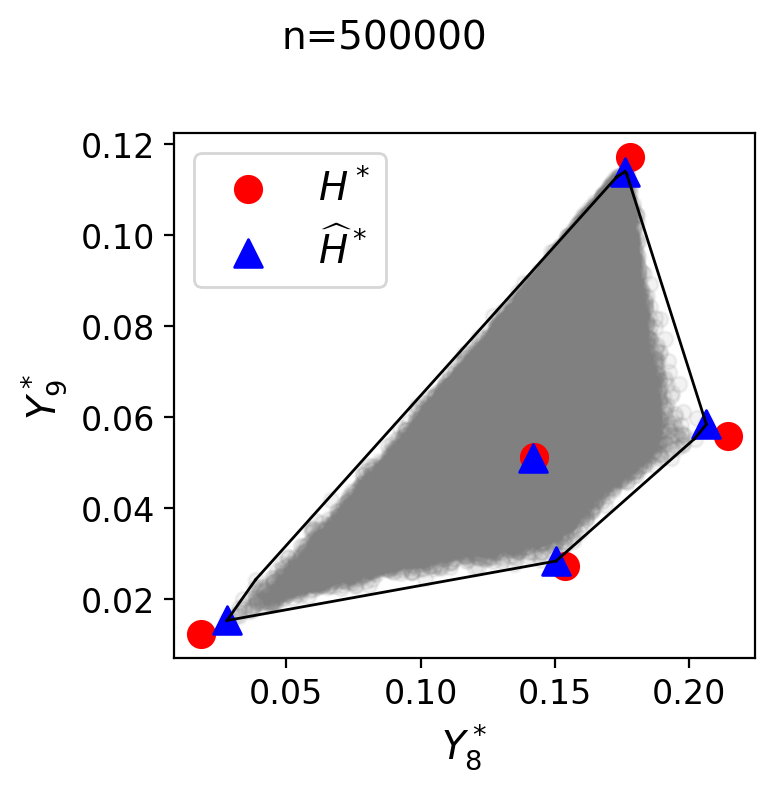}
    \caption{Scatter plot of a randomly selected pair of variables from $Y^*$ with the sample convex hull, overlaid with the true $H^*$ (red circles) and the estimated $\widehat{H}^*$ (blue triangles), for randomly selected replicates at $n = 300$ (left), $n = 100000$ (middle), and $n = 500000$ (right).}
    \label{fig:sim_iidW_H}
\end{figure}

\subsection{Stationary ergodic process $W$}
\label{sec:sim_ar1W}

The next set of simulations demonstrates robustness of the consistency result under a more complex data-generating process. The emissions $W$ follow a stationary ergodic process, and the observations are subject to multiplicative measurement errors, constituting a misspecification of \eqref{eq:main}. With $J=8$ and $K=3$, latent data are drawn from an NMF model $\widehat{Y} = WH$, and observed data are generated via $Y_{ij} = \epsilon_{ij}\widehat{Y}_{ij}$, where the measurement error satisfies $\log \epsilon_{ij} = \log a_i + \log b_j + \tau\eta_{ij}$. The parameter $\tau$ governs the degree of non-separability: when $\tau = 0$, the model reduces to a purely separable log-measurement error structure. We draw $u_i := \log a_i \sim N(-\sigma_u^2/2,\, \sigma_u^2)$, ensuring $\mathbb{E}(a_i) = 1$, and independently set $v_j := \log b_j \sim N(0, \sigma_v^2)$ and $\eta_{ij} \sim N(0, 1)$. Throughout, we fix $\sigma_u = \sigma_v = 0.3$ and vary $\tau \in \{0, 0.1, 0.3\}$ to assess sensitivity to departures from separability.

We generate $W$ from a log-AR process of order one (log-AR(1)). For each $k\in\{1,\dots,K\}$, we simulate a stationary Gaussian AR(1) on $g_{ik}=\log W_{ik}$:
$$
g_{ik} = \mu^{(g)}_k + \phi_k\big(g_{i-1,k}-\mu^{(g)}_k\big) + \varepsilon_{ik}, 
\qquad \varepsilon_{ik}\sim N(0,\sigma_{\varepsilon,k}^2),\quad |\phi_k|<1,
$$
initialized at stationarity $g_{1k}\sim N\!\big(\mu^{(g)}_k,\ \sigma_{\varepsilon,k}^2/(1-\phi_k^2)\big)$. Exponentiation yields non-negative factors $W_{ik}=\exp(g_{ik})$ with stationary lognormal marginals. The corresponding population mean is 
$$
\mu_k = \mathbb E(W_{ik}) = \exp\!\Big(\mu^{(g)}_k + \tfrac12\,\tfrac{\sigma_{\varepsilon,k}^2}{1-\phi_k^2}\Big).
$$ The autoregressive coefficient is fixed at $\phi_k = 0.8$ for all $k$, imposing strong temporal correlation in the data. The location and scale parameters of the latent Gaussian innovations are randomly drawn as $\mu^{(g)}_k \sim \mathrm{Unif}(-0.5, 0.5)$ and $\sigma_{\varepsilon,k} \sim \mathrm{Unif}(0.15, 0.5)$, respectively. All other settings mirror Section \ref{sec:sim_iidW}.

For $n \in \{100, 300\}$, we use exact exhaustive search. For $n \in \{1500, 10000\}$, we prune hull candidates to approximately $50K$ and apply greedy refinement. For $n \in \{100000, 500000\}$, we first extract up to $500K$ directional extremes along 20000 randomly generated directions, prune these to approximately $50K$ candidates, and then apply greedy refinement. The pruned candidate sets are larger than in Section \ref{sec:sim_iidW} because $K=3$ keeps the problem tractable even at this scale. 

Figure \ref{fig:sim_ar1W_nfd} summarizes NFD across replicates for varying $n$ and $\tau$. Under fully separable or mildly non-separable errors ($\tau = 0, 0.1$), estimation error for $\widehat{\Phi}$ decreases monotonically with $n$, consistent with our consistency result. Under moderate non-separability ($\tau = 0.3$), the NFD of $\widehat{H}^*$ increases with $n$, suggesting that the multiplicative errors distort the data cloud and increasingly obscure the true vertices. The NFD of $\widehat{\Phi}$, however, plateaus, indicating that the source attribution estimates remain stable despite the misspecification. NRMSE results are qualitatively similar and are omitted.

\begin{figure}
    \centering
    \includegraphics[width=0.98\linewidth]{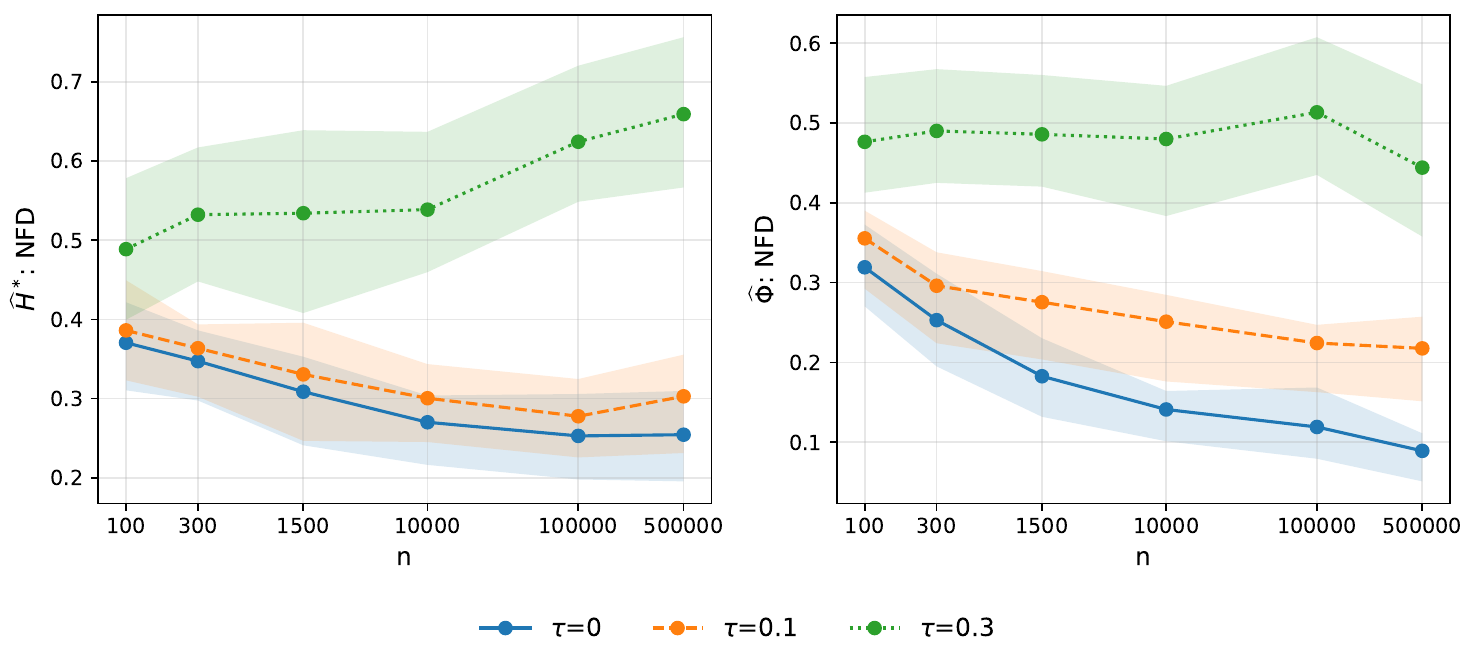}
    \caption{NFD of $\widehat{H}^*$ and $\widehat{\Phi}$ as a function of $n$, for $\tau \in \{0, 0.1, 0.3\}$, based on 50 replicates under the stationary ergodic $W$ setting. Lines show the mean and shaded bands the interquartile range.}
    \label{fig:sim_ar1W_nfd}
\end{figure} 

Figure \ref{fig:sim_ar1W_heat} compares true $\Phi$ with $\widehat{\Phi}$ for randomly selected replicates at $n \in \{300, 100000, 500000\}$ under the fully separable setting ($\tau = 0$). The visual agreement between the two heatmaps improves with $n$, illustrating the convergence suggested by Figure \ref{fig:sim_ar1W_nfd}.

\begin{figure}
    \centering
    \includegraphics[width=0.98\linewidth]{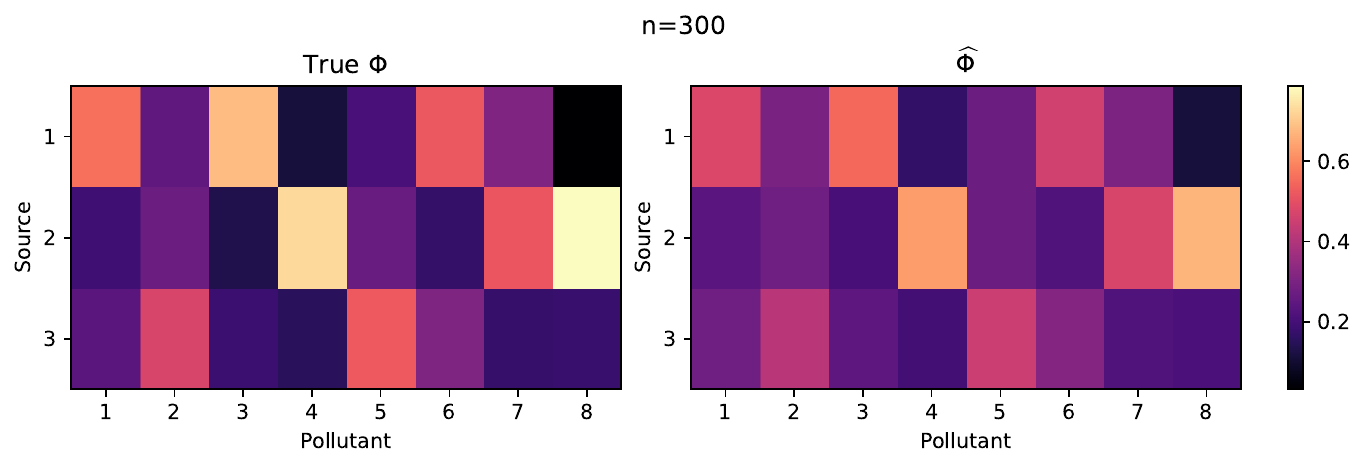}
    \includegraphics[width=0.98\linewidth]{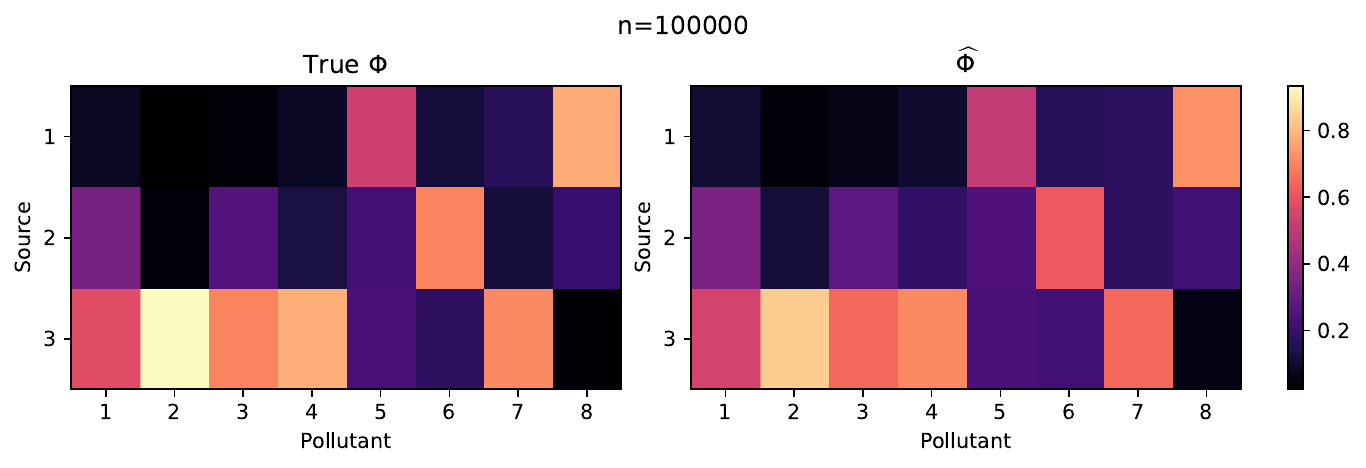}
    \includegraphics[width=0.98\linewidth]{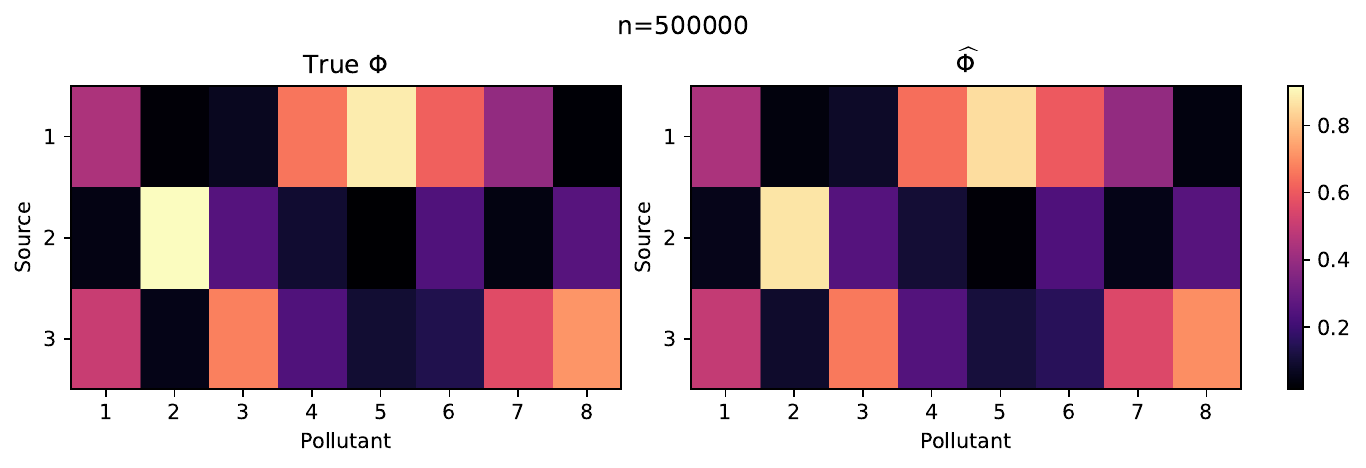}
    \caption{Heat maps of the true $\Phi$ and the estimate $\widehat \Phi$ for a randomly selected replicate under $\tau=0$, for $n=300$ (top), $n=100000$ (middle), and $n=500000$ (bottom).}
    \label{fig:sim_ar1W_heat}
\end{figure}

\newpage
\section{Source apportionment of Arctic aerosols} \label{subsec:GL}

\subsection{Preprocessing and suitability of the data} \label{subsec:GL_prep}

We preprocessed the raw record as follows. First, we retained reported concentrations below the instrumental detection limit as measured, rather than replacing them with a substituted value. The detection limit bounds where a single measurement can be confidently distinguished from background, not the range over which the instrument response tracks concentration. Concentrations below the detection limit therefore retain real information about relative magnitude and covariation. Substitution collapses the low tail onto a constant, erasing covariation among elements and distorting the data geometry. This choice is consequential because most of the selected elements fall below detection in more than 10\% of samples, and Ni in as many as 40\%. Second, we removed observations for which the reported measurement uncertainty, provided alongside each concentration in the dataset, exceeded 5 $\mathrm{ng\,m^{-3}}$. This threshold lies well beyond the 99th percentile of reported uncertainties (1.855) and eliminates a small number of spuriously high values identified by visual inspection. Third, to reduce missingness, we aggregated the three smallest size bins, [0.09, 0.26), [0.26, 0.34), and [0.34, 0.56), into the single `fine' bin [0.09, 0.56) $\upmu\mathrm{m}$, and the three largest bins, [1.15, 2.5), [2.5, 5), and [5, 10), into the single `coarse' bin [1.15, 10) $\upmu\mathrm{m}$, summing concentrations and combining uncertainties in quadrature. The aggregated bins tend to correlate strongly due to similar origination processes, i.e., gas-particle conversion and mechanical grinding followed by aeolian resuspension, respectively. Fourth, non-positive concentrations, which are physically impossible and reflect measurement artifacts, were truncated to the element-specific minimum positive value, since non-negativity is intrinsic to receptor models. The percentage of negative values varies by element: Cu (0.3\%), Ni (3.8\%), Pb (9.7\%), S (6.6\%), Zn (0.9\%), Al (0.4\%), Ca (3.7\%), Cl (5.8\%), Si (10.0\%), and Ti (10.1\%). Zero reported uncertainties were replaced by the element-specific minimum positive uncertainty. This adjustment is required only by ESAT \citep{smith_esat_2024}, the Python implementation of PMF used for comparison, which takes per-observation uncertainties as input. GeomNMF and the other comparison methods, the N-FINDR-style greedy search and LS-NMF, require no uncertainty input. Finally, although all elements share the same measurement unit, their dynamic ranges differ by orders of magnitude (Figure \ref{fig:GL_raw_timeseries}, Supplementary Material \ref{subsec:GL}). We therefore normalized each element by its range between the 95th percentile and the minimum. Using the 95th percentile rather than the maximum makes the rescaling robust to extreme values. The final dataset comprises 2025 (out of 4780) days of measurements between May 17, 2003 and June 16, 2016. Data gaps appeared frequently due to either operational outages at the site or instrument malfunctions. The laboratory analytical methodology did not change during the selected time period. Despite these frequent gaps, all calendar months are represented (Table \ref{tab:GL_month_coverage}, Supplementary Material \ref{subsec:GL}). The record after mid-2016 remained too sparse in the selected elements to include, even after the bin aggregation. 

The Summit DRUM record exercises several features of our framework. First, the element-specific rescaling applied above is exactly the class of transformation to which $\Phi$ is invariant by construction (Theorem \ref{thm:welldefn}). Second, the analysis proceeds on 10 of the 29 measured elements. As discussed in Section \ref{sec:est_phi}, GeomNMF consistently estimates the submatrix of $\Phi$ restricted to a selected subset of elements, so the pre-selection of the tracer elements carries no cost in identifiability and consistent estimation. Third, source intensities at Summit are expected to be temporally dependent, which our stationarity and ergodicity assumption (Assumption \ref{asm:erg}) accommodates without parametric distributional assumptions. Fourth, below-detection observations are pervasive rather than incidental at Summit (Figure \ref{fig:GL_raw_timeseries}, Supplementary Material \ref{subsec:GL}), reflecting the genuinely low concentrations characteristic of a remote Arctic site rather than a deficiency of the instrument. This poses a difficulty specific to PMF-type methods, which require a specified uncertainty for every observation and conventionally handle below-detection values through substitution and uncertainty-inflation rules \citep{polissar_atmospheric_1998, dutton_source_2010}. Low signal-to-noise variables are then downweighted or discarded to prevent them from inflating errors in the recovered factors \citep{paatero_discarding_2003}. At Summit, however, this prescription is difficult to apply because the low-concentration observations are not noise to be suppressed but the dominant mode of the data, and the degree of downweighting is neither physically motivated nor determined by any principle. 

\subsection{Implementation details} \label{subsec:GL_detail}

GeomNMF computed the exact convex hull from the data and searched over pruned candidates followed by greedy refinement for scalability. The number of pruned candidates was approximately $10K$ for $K = 6$ and $20K$ for $K = 5$. The N-FINDR-style greedy search was initialized at ATGP. LS-NMF was fit using NMF in the Python package scikit-learn with 1000 maximum iterations and otherwise default settings. The regularized LS-NMF used regularization constants of 0.01 for both $W$ and $H$ (\texttt{alpha\_W = 0.01}, \texttt{alpha\_H = 0.01}). ESAT was run using the code at \url{https://github.com/quanted/esat}. Default settings were used except for non-normalized input (\texttt{init\_norm=False}), a convergence threshold of 0.01 (\texttt{converge\_delta=0.01}), and 50 iterations below the threshold required to count as converged (\texttt{converge\_n=50}). ESAT includes a built-in bootstrap function, so seeds could not be matched with the other methods.

\subsection{HYSPLIT details} \label{subsec:hysplit}

To aid source interpretation, we ran the Hybrid Single-Particle Lagrangian Integrated Trajectory model (HYSPLIT; \citealt{stein_noaas_2015}), which simulates atmospheric transport pathways, for each day with a DRUM sample. We considered two meteorological inputs: the 1-degree Global Data Assimilation System (GDAS1) and the 2.5-degree National Centers for Environmental Prediction and National Center for Atmospheric Research (NCEP/NCAR) reanalysis. The two meteorology inputs produced very similar trajectories. The reanalysis covered more of the sampling period than GDAS1 (1861 versus 1666 of 2025 days), so we report reanalysis results throughout. For each sampled day we computed seven-day back trajectories initialized at Summit at 20, 50, 100, and 650 m above ground level and at 00:00, 06:00, 12:00, and 18:00 UTC. The 650 m height was chosen to remain above the boundary layer at Summit. Trajectories at 20, 50, and 100 m were nearly identical, suggesting that they are all frequently within or influenced by the boundary layer, whereas the 650 m trajectories revealed transport from farther away. The two sets nonetheless followed similar source pathways, indicating that the near-surface levels at which the DRUM samples are collected partly reflect the free-tropospheric air sampled aloft. The reported figures pool trajectories across all release heights and initialization times. HYSPLIT was run through the \texttt{splitr} R package (\url{https://github.com/rich-iannone/splitr}).

\subsection{Figures and tables}

\begin{figure}
    \centering
    \includegraphics[height=0.38\paperheight]{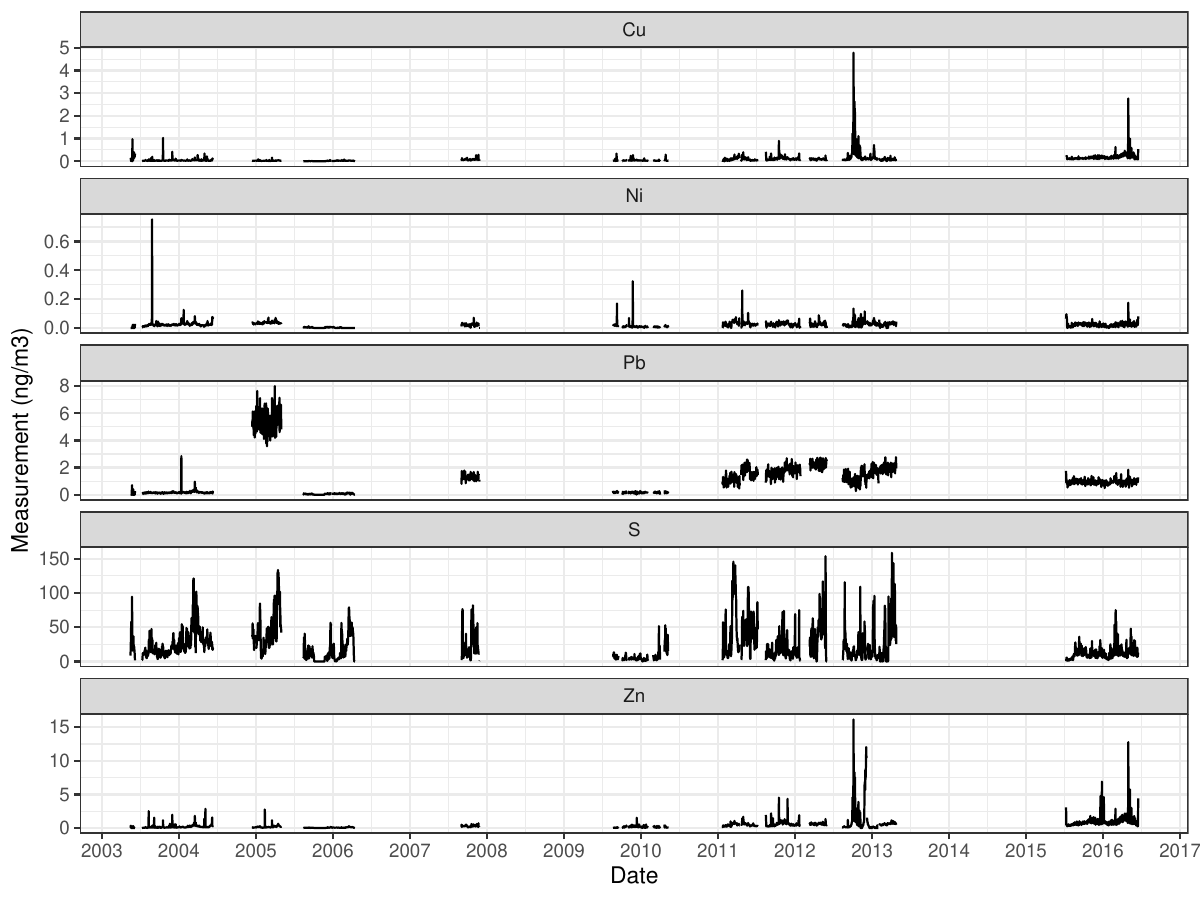}\\
    \includegraphics[height=0.38\paperheight]{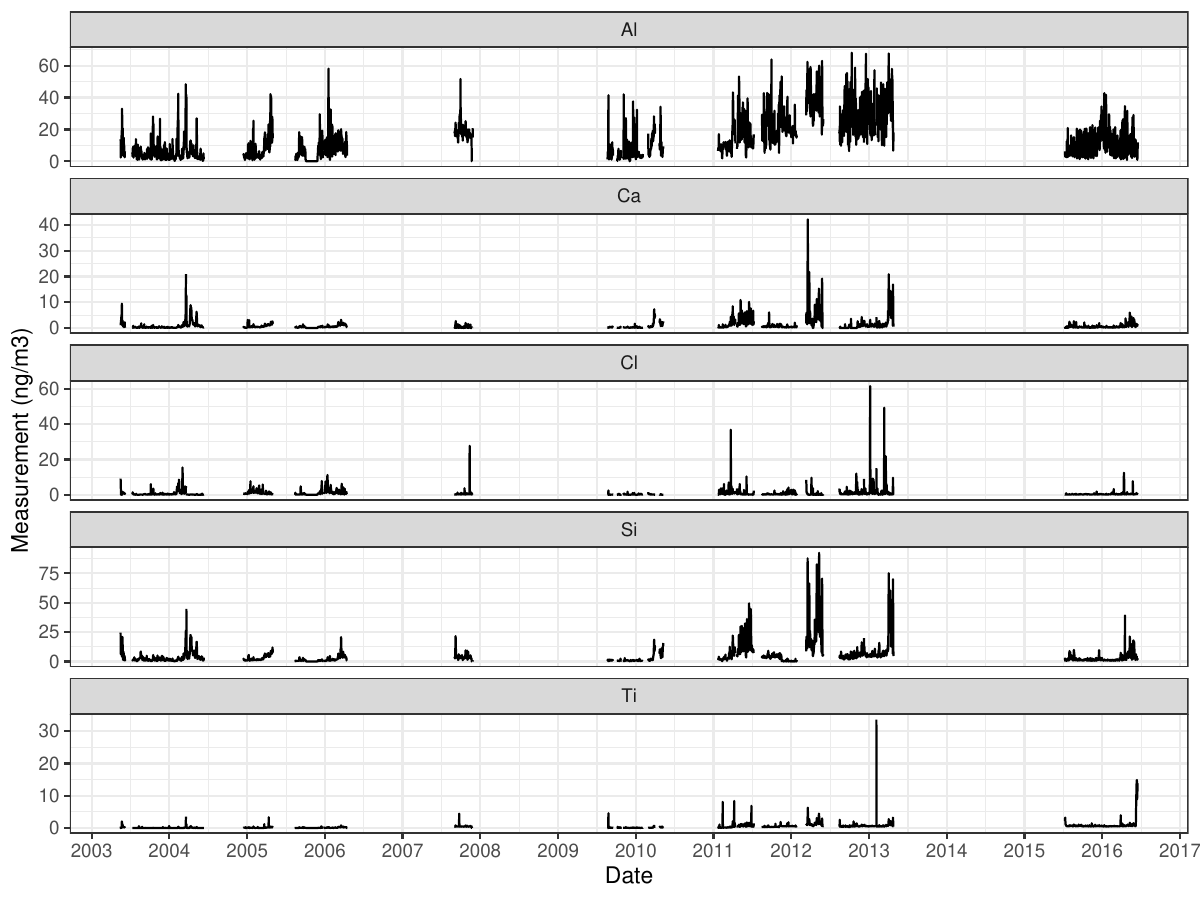}
    \caption{Timeseries of DRUM measurements of five anthropogenic elements in [0.09, 0.56) $\mu m$ (top) and five natural elements in [1.15, 10) $\mu m$ (bottom) at the Greenland Summit site.}
    \label{fig:GL_raw_timeseries}
\end{figure}

\begin{table}[htbp]
\centering
\caption{Monthly coverage of the final Greenland Summit dataset ($n=2025$ days, May 17, 2003 -- June 16, 2016). All calendar months are represented, with reduced sampling in June and July reflecting melt-season constraints.}
\vspace{5pt}
\begin{tabular}{lcccccccccccc}
\toprule
Month & Jan & Feb & Mar & Apr & May & Jun & Jul & Aug & Sep & Oct & Nov & Dec \\
\midrule
Days  & 223 & 171 & 237 & 195 & 150 & 64  & 52  & 126 & 195 & 208 & 204 & 200 \\
\bottomrule
\end{tabular}
\label{tab:GL_month_coverage}
\end{table}

\begin{table}[htbp]
\centering
\caption{Diagnostics for $\Phi$ stability across $K$. CVar is the mean cellwise coefficient of variation (bootstrap SE divided by bootstrap mean); lower values indicate more stable estimates. $R$ is the mean element-specific rank stability; higher values indicate more consistent identification of the top contributor.}
\begin{tabular}{lcc}
\toprule
 & CVar & $R$ \\
\midrule
$K=5$ & 0.090 & 0.872 \\
$K=6$ & 0.098 & 0.891 \\
\bottomrule
\end{tabular}
\label{tab:phi_stability_K}
\end{table}

\begin{figure}[htbp]
    \centering
    \includegraphics[width=0.85\textwidth]{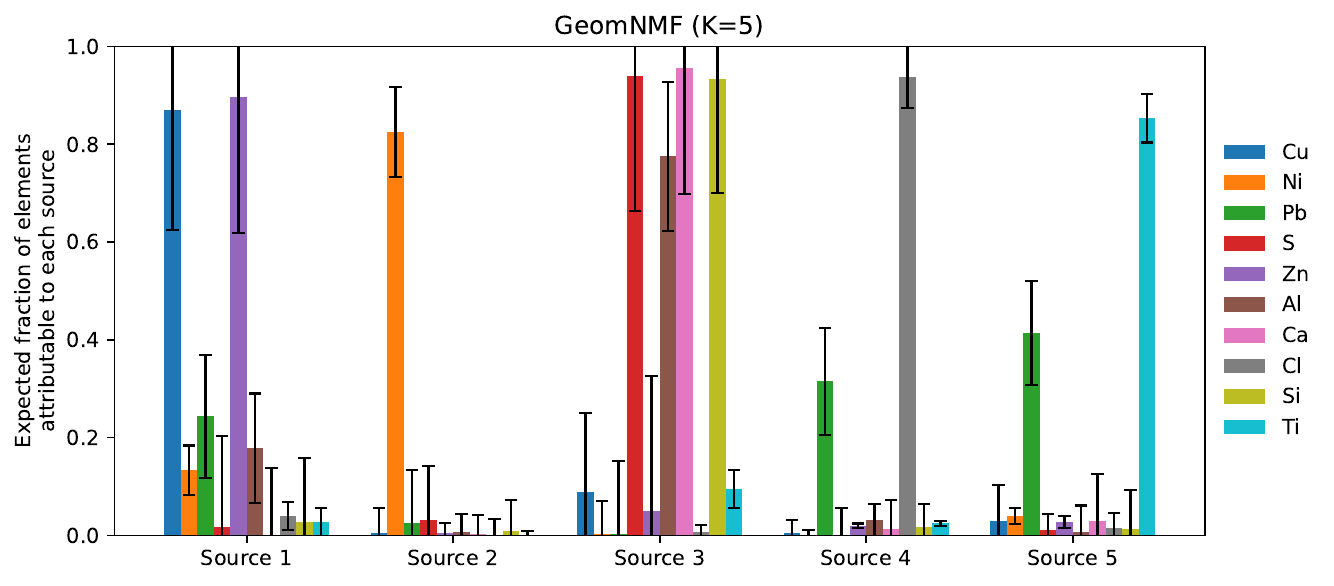}
    \includegraphics[width=0.85\textwidth]{figure/GL_scaled_K6_minK60_Phi_hat_bootCI}
    \caption{Estimated source attribution matrix $\widehat{\Phi}$ from GeomNMF with $K=5$ (top) and $K=6$ (bottom). The bottom panel repeats Figure \ref{fig:GeomNMF_K6_Phi} for direct comparison.}
    \label{fig:GeomNMF_K5_K6_Phi}
\end{figure}

\begin{table}[htbp]
\centering
\caption{Stationarity tests applied to the estimated source intensities from the six-source GeomNMF. The ADF test takes a unit root (non-stationarity) as the null and the KPSS test takes level stationarity as the null hypothesis, so the two tests provide complementary evidence. KPSS $p$-values are truncated at the boundaries of the reference table, so $<0.01$ and $>0.10$ indicate only that the statistic falls outside the tabulated range.}
\vspace{5pt}
\label{tab:stationarity}
\begin{tabular}{lccl}
\hline
Source & ADF $p$ & KPSS $p$ & Reading at $\alpha=0.05$\\
\hline
Source 1 & $<0.001$ & $<0.01$ & inconclusive \\
Source 2 & 0.017 & $>0.10$ & stationary \\
Source 3 & $<0.001$ & 0.072 & stationary \\
Source 4 & $<0.001$ & 0.068 & stationary \\
Source 5 & $<0.001$ & $>0.10$ & stationary \\
Source 6 & 0.051 & $<0.01$ & inconclusive \\
\hline
\end{tabular}
\end{table}



\begin{figure}[htbp]
    \centering
    \includegraphics[height=0.23\textheight]{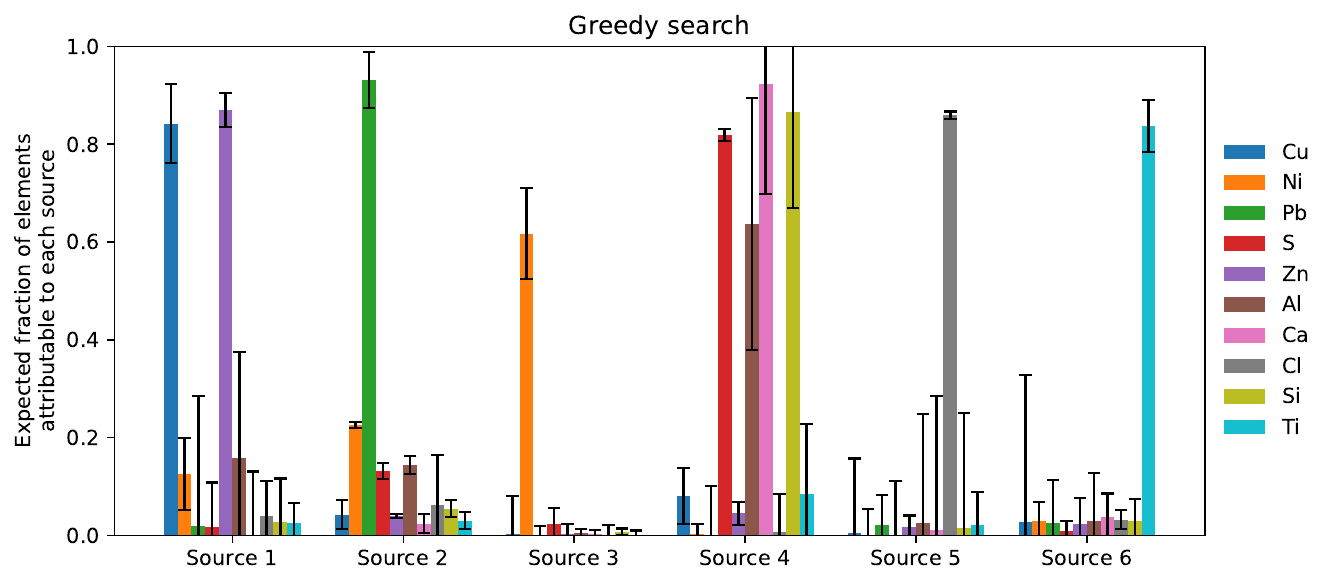}
    \includegraphics[height=0.23\textheight]{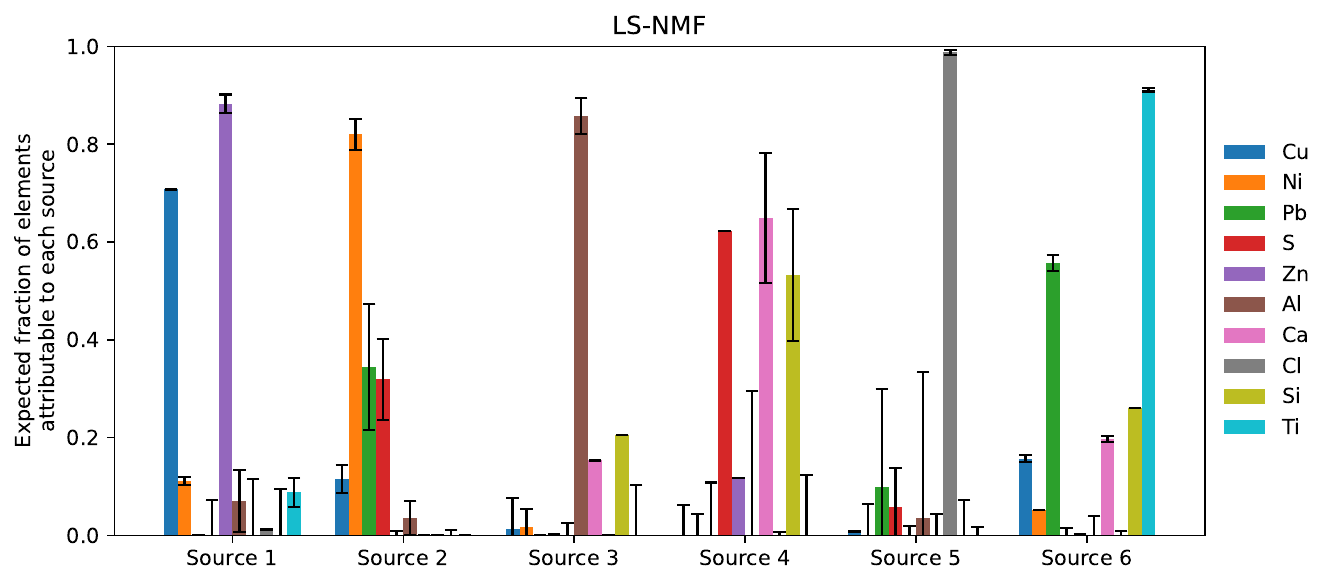}
    \includegraphics[height=0.23\textheight]{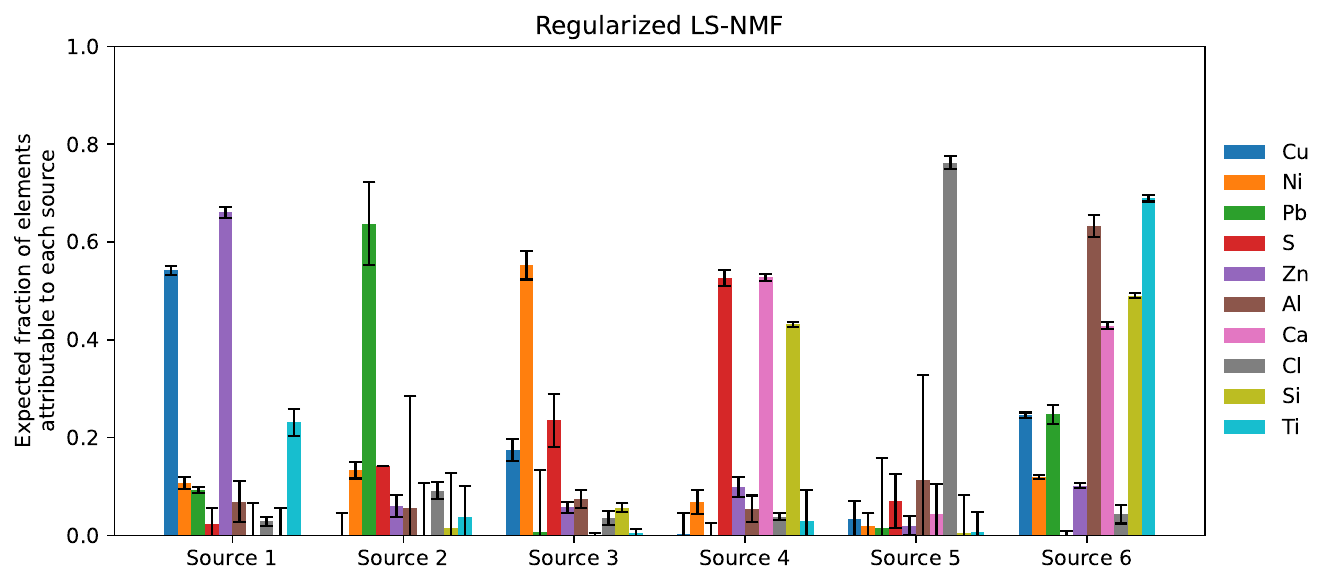}
    \includegraphics[height=0.23\textheight]{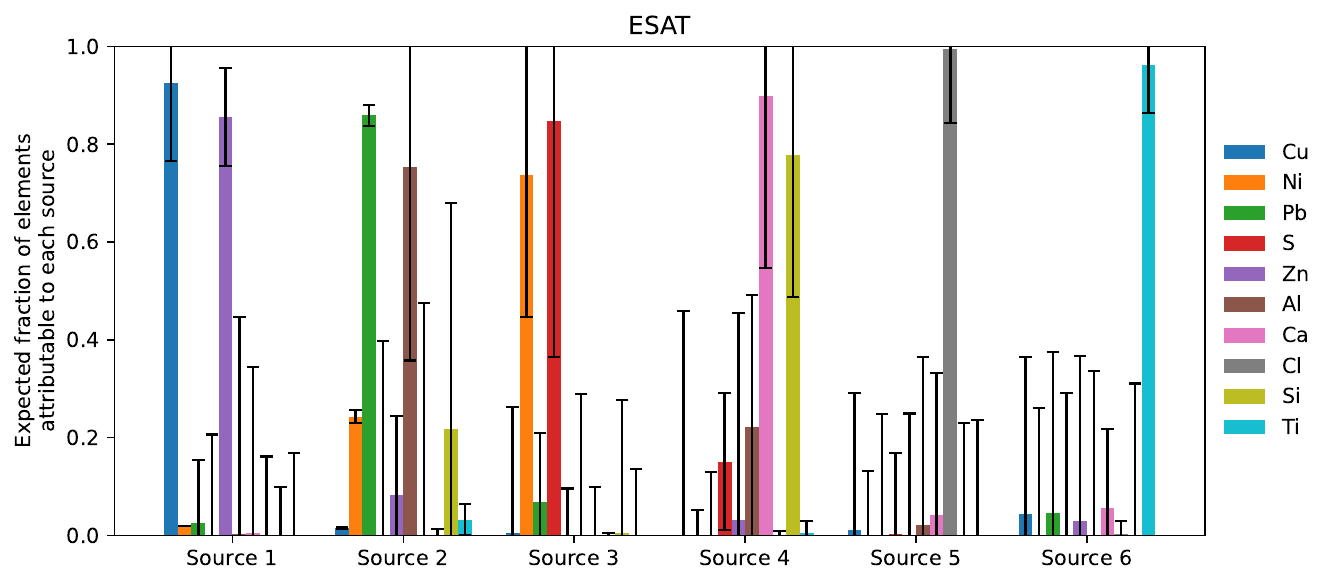}
    \caption{Estimated source attribution matrix $\widehat{\Phi}$ from the greedy search, unregularized LS-NMF, regularized LS-NMF, and ESAT, top to bottom. Bars show the point estimate from the data, and vertical segments show the bootstrap SE.}
    \label{fig:alternative_K6_Phi}
\end{figure}

\begin{figure}[htbp]
    \centering
    \includegraphics[width=0.8\textwidth]{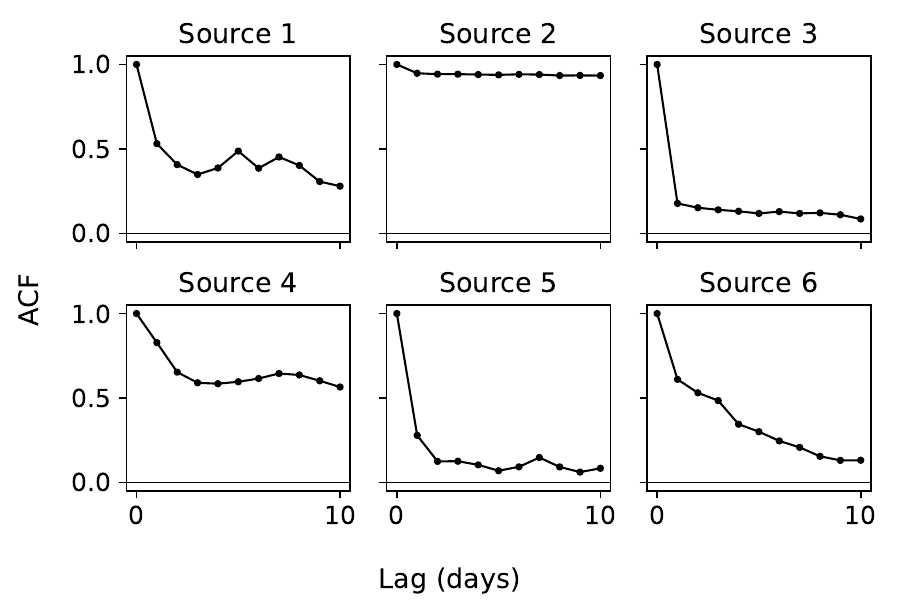}
    \caption{Autocorrelation function (ACF) of the six GeomNMF source intensities, for lags up to 10 calendar days.}
    \label{fig:GeomNMF_W_lag10acf}
\end{figure}

\begin{figure}[htbp]
    \centering
    \includegraphics[width=0.7\textwidth]{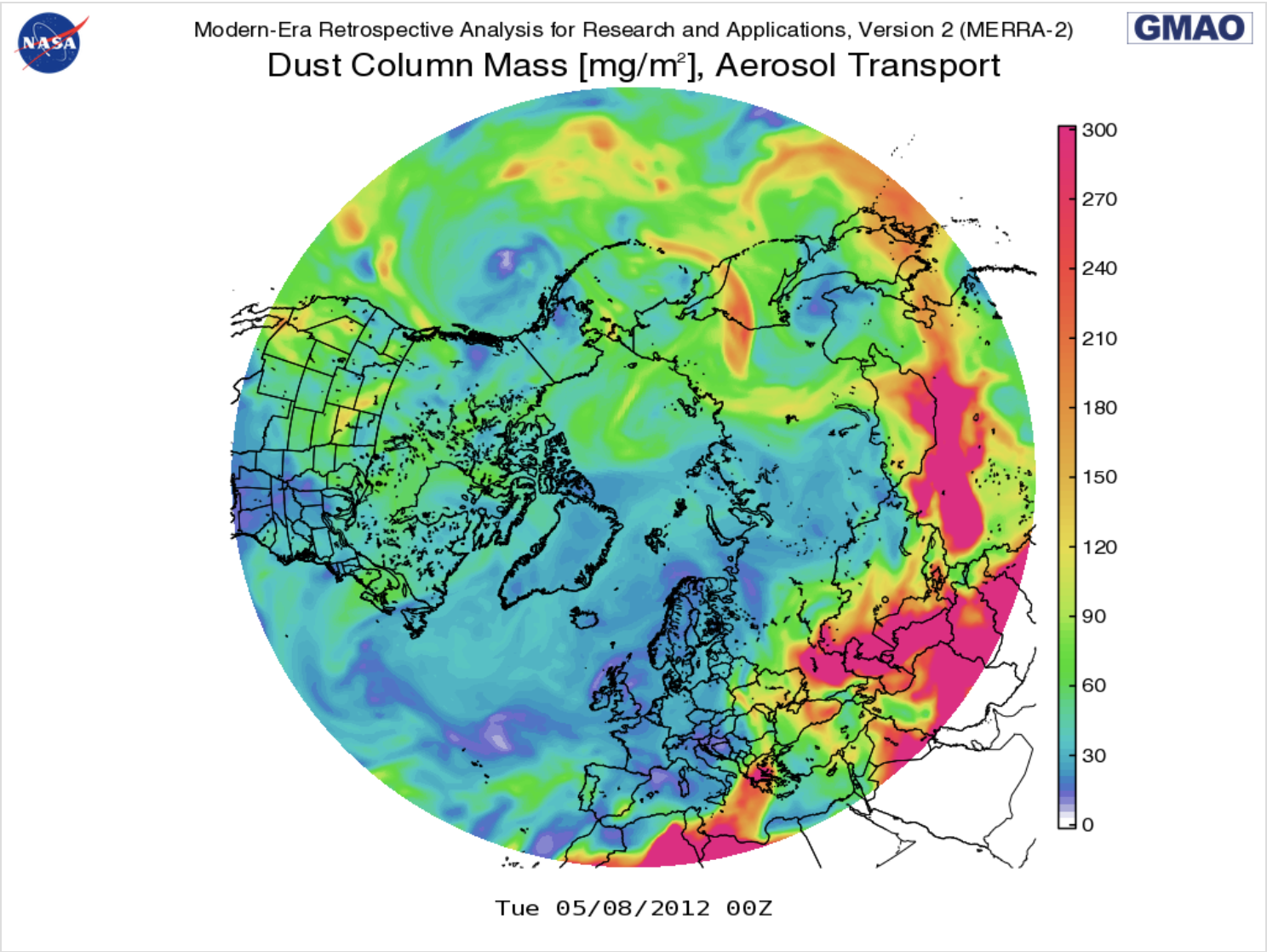}
    \caption{MERRA-2 dust column mass on May 8, 2012, an example day on which dust from a Source 4-trajectory is in transit between the Gobi Desert and Summit. Elevated dust column mass extends across northern Eurasia along the trajectory pathway.}
    \label{fig:source4_gobi_supporting}
\end{figure}

\newpage 

\bibliographystyle{style/BJ_apalike}
\bibliography{text/bibliography}

\end{document}